\documentclass[12pt,a4paper]{article}
\usepackage{amssymb,amsmath,amsthm,amscd,graphicx}
\usepackage[toc,page]{appendix}

\usepackage{xcolor}
\usepackage{fouridx}
\usepackage{mathrsfs,mathtools}
\usepackage{yfonts}
\mathtoolsset{showonlyrefs}

\newcommand{\dualV}[2]{{\fourIdx{}{V^\ast}{}{V}{\langle #1, #2 \rangle}}}

\newcommand{\sol}{\mathrm{sol}}
\newcommand\divv{{\rm div}\,}

\newtheorem{theorem}{Theorem}[section]

\newtheorem{proposition}[theorem]{Propostion}

\newtheorem{lemma}[theorem]{Lemma}
\newtheorem{definition}[theorem]{Definition}
\newtheorem{remark}[theorem]{Remark}

\newcommand{\norm}[3]{\|  #1 {\| }_{#2}^{#3}}
\newcommand{\lb}{\,\langle}
\newcommand{\rb}{\rangle\,}
\newcommand{\dirilsk}[3]{{\bigl( \! \bigl( #1 , #2 \bigr) \! \bigr)}_{#3}}
\newcommand\dela[1]{}
\newcommand{\embed}{\hookrightarrow }
\newcommand{\dist}{\mathrm{dist}}

\newcommand{\Var}{\mathrm{Var}}
\newcommand{\rH}{\mathrm{H}}
\newcommand{\rK}{\mathrm{K}}
\newcommand{\rV}{\mathrm{V}}

\DeclareMathOperator{\Leb}{\mathrm{Leb}}

\usepackage{todonotes}

\numberwithin{equation}{section}

\begin{document}
\title{Reflection of Stochastic Evolution Equations in Infinite Dimensional Domains}
\author{ Zdzis{\l}aw Brze{\'z}niak$^{1}$ and Tusheng Zhang$^{2}$,  }

\footnotetext[1]{\ Department of Mathematics, University of York, Heslington, YO10 5DD,
York, United Kingdom, email: zdzislaw.brzezniak@york.ac.uk}
\footnotetext[2]{\ School of Mathematics, University of Science and Technology of China,
Hefei, Anhui, China,  email: tusheng.zhang@manchester.ac.uk }
 \maketitle

{\bf Abstract}: In this paper, we establish the  existence and  the uniqueness
of solutions of stochastic evolution equations (SEEs) with reflection in an infinite dimensional ball. Our framework is sufficiently general to include e.g. the stochastic
Navier-Stokes equations.



\section{Introduction}
\setcounter{equation}{0} Let $D:=B(0,1)$ be the open unit ball in a separable Hilbert space $H$ endowed with an inner product denoted by $(\cdot,\cdot)$.
Let $A$ be a self-adjoint, positive definite operator on the Hilbert space $H$ and let $B$ be a certain unbounded bilinear map from $H\times H$ to $H$.

 In this article we  consider the following stochastic evolution equations (SEEs) with reflection. To be more specific,
 given a filtered probability space $\mathfrak{P}=(\Omega,
\mathcal{F},\mathbb{F},\mathbb{P})$,  satisfying the so-called usual condition,  a real-valued Brownian Motion  $W=(W(t):\, t\geq 0)$    defined on $\mathfrak{P}$,
and an element $u_0\in\bar{D}$ we will be looking for a pair $(u, L)$ which solves, in the sense that will be made precise below,  the following initial value  problem
 \begin{align}\label{original equation}
du(t)&+Au(t)\,dt=f(u(t))\,dt+B(u(t),u(t))\,dt+\sigma(u(t))\,d{W}(t)+dL(t),\;\;t\geq 0,\\
\label{eqn-ic}
u(0)&=u_0,
\end{align}
where $u$ and $L$ are  respectively $D$ and $H$-valued adapted stochastic processes,  with continuous,  and respectively, of locally  bounded variation,   see \cite[\S 17 and Theorem III.2.1, p. 358]{Dinculeanu_1967},
trajectories.
 We assume that the coefficients $f$ and $\sigma$ are
measurable maps from $H$ to $H$ and that  $B: V\times V\rightarrow V^\ast$ is a bilinear measurable function  whose properties  will be specified later. Here $V=D(A^{\frac{1}{2}})$, by $V^\ast$ we denote  the dual of $V$, we identify $H^\ast$, the dual of $H$, with $H$ so that  we have a Gelfand triple
\begin{equation}
\label{eqn-Gelfand triple}
V \embed H = H^\ast \embed V^\ast .
\end{equation}
  The following is the
definition of a solution to Problem \eqref{original equation}-\eqref{eqn-ic}.
\begin{definition}\label{def-solution} A pair $(u,L)$ is said to
be a solution of Problem (\ref{original equation})-\eqref{eqn-ic} iff the following conditions are satisfied
\begin{trivlist}
\item[(i)] $u$ is a $ \bar{D}$-valued continuous  and  $\mathbb{F}$-progressively measurable  stochastic process with $u\in L^2([0, T], V)$, for any $T>0$,   $\mathbb{P}$- a.s.
\item[(ii)] the corresponding $V$-valued process is strongly $\mathbb{F}$-progressively measurable;
\item[(iii)] $L$ is a $H$-valued,  $\mathbb{F}$-progressively measurable  stochastic process of paths of  locally bounded variation
such that $L(0)=0$ and
 \begin{equation}\label{eqn-variation finite}
 \mathbb{E}\,\bigl[ \vert \Var_{H}(L)([0, T]) \vert^2 \bigr]<+\infty,\;\;\; T\geq 0,
 \end{equation}
  where, for a function $v:[0, \infty)\to H$,  $\Var_{H}(v)([0, T])$ is  the total variation of $v$ on $[0, T]$ defined by
 \begin{equation}\label{eqn-variation}
 \Var_{H}(v)([0, T])]:= \sup  \sum_{i=1}^n \vert v(t_i)-v(t_{i-1})\vert,
 \end{equation}
 where the supremum is taken over all partitions $ 0=t_0<t_1<\ldots < t_{n-1}<t_n=T$, $ n\in \mathbb{N}$, of the interval $[0,T]$;
\item[(iii)] $(u, L)$ 
satisfies the following integral identity in  $V^\ast$,  for every $t\geq0$, $\mathbb{P}$-almost surely,

\begin{equation}\label{eqn-SEE with reflection}
\begin{aligned}
 u(t) &+\int_0^tAu(s)\,ds-\int_0^tf(u(s))\,ds-\int_0^t B(u(s),u(s))\,ds\\
&=u(0)+\int_0^t\sigma(u(s))\,dW(s)+L(t);
\end{aligned}
\end{equation}

\item[(iv)]  for every  $T>0$,  and $\phi\in C([0,T], \bar{D})$, $\mathbb{P}$-almost surely,
\begin{equation} \int_0^T ( \phi(t)-u(t), L(dt)) \geq 0.
\end{equation}
where the integral on the LHS is the  Riemiann-Stieltjes integral  of the $H$-valued function $\phi-u$ with respect to an $H$-valued bounded-variation function $L$, see
\cite[p. 47 in section 1.3.3]{Barbu+Precupanu_2012}.
\end{trivlist}
\end{definition}
\begin{remark}\label{rem-one dim BM}
Here we choose to have a one-dimensional Brownian motion  for the simplicity of the exposition. The method in this paper works for infinite dimensional Brownian motion as well.
Let us emphasize that  we do not allow the so-called gradient noise because the map $\sigma$ maps $H$ to $H$ and not $V$ to $H$. We still believe that our results are true in  the latter case under some additional coercivity assumptions, see e.g.
\cite{Brz+Mot_2013} and \cite{Brz+Dh_2018}. However, we have not verified the details. 
\end{remark}
The existence and the uniqueness of solutions of  real-valued stochastic partial differential equations (SPDEs) with reflection at zero driven by space-time white noise were obtained by Nualart and Pardoux in \cite{NP}
  for  additive noise, by Donati-Martin and Pardoux
 in \cite{MP} for general diffusion coefficient $\sigma$ without proving the uniqueness and by T. Xu and T. Zhang in \cite{XZ} for general $\sigma$ with also the proof of the uniqueness.
Various properties of the solution of the real-valued SPDEs with reflection were studied  in \cite{DMZ},\cite{DP},\cite{HP},
\cite{ZA} and \cite{Z}. SPDEs with reflection can  be used to
model the evolution of random interfaces near a hard wall. It was
proved by T. Funaki and S. Olla in \cite{FO} that the fluctuations
of
a $\nabla \phi$ interface model near a hard wall converge in law to the stationary solution of a SPDE with reflection. For stochastic Cahn-Hilliard equations with reflection, please see \cite{DebZamb}.

\vskip 0.3cm

 \indent The purpose of this paper is to establish the existence and the
uniqueness of the reflection problem of stochastic evolution equation on an infinite dimensional ball in a separable Hilbert space. Under this setting, we would like to mention two related papers. In \cite{BD-1} the authors considered a reflection problem for 2D stochastic Navier-Stokes equations with periodical boundary conditions in an infinite dimensional ball. The problem is formulated as a stochastic variational inequality, the Galerkin approximations and the Kolmogorov equations are used. A stochastic reflection problem was considered in \cite{BDT} for  stochastic evolution equations driven by additive noise on a closed convex subset in a Hilbert space. The solution is defined as a solution to a control problem.The approach is very much analytical.
\vskip 0.3cm
In this paper, we consider the reflection problem for general stochastic evolution equations with mutiplicative noise on an infinite dimensional ball in a Hilbert space. The approach  is stochastic and direct.
We will solve the reflection problem using
approximations of penalized stochastic evolution equations. To prove  convergence of the solutions of the approximating equations,  due to the lack of the comparison theorems, we need to obtain a number of good estimates for the solutions of the penalized equations. Our approach is inspired by the work in \cite{HP}. Our framework is quite general, which includes also stochastic Navier-Stokes equations.
\vskip 0.3cm
The rest of the paper is organised as follows. In Section \ref{sec-2}, we describe the setup and state the precise assumptions on the coefficients. In Section \ref{sec-approximation}, we consider approximating penalized stochastic evolution equations and obtain a number of estimates for the solution which are used later. Section \ref{sec-main} is devoted to the statement and proof of our main result, i.e. Theorem \ref{thm-main}, about the existence and the uniqueness of the stochastic reflection problem. In Section \ref{sec-Examples} we discuss a concrete example of damped Navier-Stokes Equations in the whole Euclidean domain $\mathbb{R}^2$.

\section{Framework}\label{sec-2}

\indent Now we introduce the setup of the paper and the assumptions. Let $H$ be a separable Hilbert space with the norm $|\cdot |_H$ or simply $|\cdot|$ and the inner product $( \cdot , \cdot)$. Let $A$ be a self-adjoint, positive definite operator on the Hilbert space $H$ such that there exists $\lambda_1>0$ such that
\begin{equation}\label{eqn-Poincare}
( Au,u) \geq \lambda_1 \vert u\vert^2, \;\;\; u \in D(A).
\end{equation}

 Set $V:=D(A^{\frac{1}{2}})$, the domain of the operator $A^{\frac{1}{2}}$. Then $V$ is a Hilbert space with the inner product
\begin{equation}\label{eqn-squarerootofA}
\dirilsk{u}{v}{}=( A^{\frac{1}{2}}u, A^{\frac{1}{2}}v),
\end{equation}
and the norm $\Vert \cdot\Vert$.
$V^\ast$ denotes the dual space of $V$. We also use $\dualV{\cdot}{\cdot}$  to denote the duality  between $V$ and $V^\ast$.
\vskip 0.3cm
We consider  two    measurable maps
\begin{align}\label{eqn-f}
  f &: H\rightarrow V^\ast   \\
  \label{eqn-sigma}
  \sigma &: H\rightarrow H
\end{align}
 a bilinear map
 \begin{align}\label{eqn-B} B &: V \times V \to V^\ast
 \end{align}
and the corresponding  trilinear form $\bar{b}: \times V\times V\rightarrow R$ defined by
\begin{align}\label{eqn-b}
\bar{b}(u,v,w)= \dualV{B(u,v)}{w}, \;\;\; u,v,w\in V.
\end{align}

\vskip 0.4cm

 We introduce the following assumptions\footnote{One can easily weaken the assumptions of the vector field $f$ as below. We will not dwell upon this in the current paper. \textit{There exists $\alpha \in [0,1)$ such that $f: D(A^{\frac{\alpha}2}) \to V^\ast$ is globally Lipschitz.} } that we will be using in the paper.
\begin{trivlist}
\item[\textbf{(A.1)}] There exists a constant $C$ such that
\begin{equation}\label{locally lipschitz}
\begin{aligned}
|f(u)-f(v)|_{V^\ast}+|\sigma(u)-\sigma(v)|_H\leq
C |u-v|_H, \quad\quad\quad \mbox{ for all } u, v\in H.
\end{aligned}\end{equation}
\vskip 0.3cm
\item[\textbf{(A.2)}] The form $\bar{b}$ satisfies the following conditions.
\begin{trivlist}
\item[a)] For all $u,w,v \in V$,
\begin{equation}\label{1.0}
\dualV{B(u,v)}{w}=\bar{b}(u,v,w)=-\bar{b}(u,w,v)=-\dualV{B(u,w)}{v}.
\end{equation}
\item[b)] For all $u,w,v \in V$,
\begin{equation}\label{1.1}
|\dualV{B(u,v)}{w}|=|\bar{b}(u,v,w)|\leq 2  \Vert u \Vert ^{\frac{1}{2}} |u|_H^{\frac{1}{2}}  \Vert w \Vert ^{\frac{1}{2}} |w|_H^{\frac{1}{2}}  \Vert v \Vert.
\end{equation}
\end{trivlist}

Let us observe that the constant $2$ could be replaced by any positive constant.

\vskip 0.3cm
\item[\textbf{(A.3)}]
Given is a $ \bar{D}$-valued  $\mathcal{F}_0$- measurable  random variable  $u_0$; here $ \bar{D}$ is the closed unit  ball in the Hilbert space $\rH$.

\end{trivlist}

\vskip 0.3cm
Let us note that  Assumption (A.2) particularly implies that
\begin{equation}\label{1.2}
\bar{b}(u,v,v)=0 \mbox{ i.e. }  \dualV{B(u,v)}{v}=0, \;\; u,v\in V
\end{equation}
and \begin{equation}\label{1.2-b}
  \Vert B(u,u) \Vert _{V^\ast}\leq 2  \Vert u \Vert  |u|_H,\;\; u\in V.
\end{equation}


Throughout the paper we assume that $\mathfrak{P}=(\Omega,\mathcal{F},\mathbb{F},\mathbb{P})$ is a filtered probability space with filtration $\mathbb{F}=\bigl(\mathcal{F}_t\bigr)_{t\geq 0}$, satisfying the usual hypothesis.
We also  assume that $W=\bigl(W(t)\bigr)_{t\geq 0}$ is an $\mathbb{R}$-valued  Wiener process on $\mathfrak{P}$. It should not be difficult to extend the results to a cylindrical Wiener process on some separable Hilbert space
$\rK$ provided that $\sigma$ is a map $H \to \gamma(\rK,\rH)$ such that the corresponding part of assumption \eqref{locally lipschitz} is replaced by
\begin{equation}\label{locally lipschitz-2}
\begin{aligned}
\Vert \sigma(u)-\sigma(v) \Vert_{\gamma(\rK,\rH)}\leq
C|   u-v |_{\rH} , \mbox{ for  all } u, v\in \rH  \mbox{ and for some }C>0.
\end{aligned}\end{equation}
Here, by $\gamma(\rK,\rH)$ we denote the space of all Hilbert-Schmidt operators from $\rK$ to $\rH$ and by
$\Vert \cdot \Vert_{\gamma(\rK,\rH)}$ we denote the corresponding Hilbert-Schmidt norm.

For a given $T>0$ by $X_T$ we will denote the following separable Banach space
\[
X_T=C([0,T];\rH)\cap L^2([0,T];\rV)
\]
endowed with the natural norm:
\begin{equation}
\Vert u \Vert^2_{X_T}:= \sup_{t \in [0,T]} | u(t)|^2_{\rH}+ \int_0^T \Vert u(s)\Vert^2\, ds.
\end{equation}

By $\mathbb{M}^2(0,T)$ we will denote the space of all $\rH$-valued continuous
$\mathbb{F}$-progressively measurable processes $u$ such that $u$ has a  $\rV$-valued $\mathcal{L}[0,T] \otimes \mathbb{F}$-measurable version $\tilde{u}$, where
by $\mathcal{L}[0,T]$  we denote the $\sigma$-field of Lebesgue measurable sets on the interval $[0,T]$, endowed with the following  norm
\begin{equation}
\label{eqn-M^2(0,T)-norm}
\Vert u \Vert^2_{\mathbb{M}^2(0,T)}:= \mathbb{E} \Bigl[ \sup_{t \in [0,T]} | u(t)|^2_{\rH}+ \int_0^T \Vert \tilde{u}(s)\Vert^2
\,ds \Bigr].
\end{equation}
It is known that $\mathbb{M}^2(0,T)$ is a separable Banach space.

Let us observe that in a non-rigorous way
\[
\mathbb{M}^2(0,T)=L^2(\Omega; \mathcal{L}[0,T]   \otimes \mathbb{F};\Leb \otimes \mathbb{P} ;X_T).
\]
Throughout the paper, $C$ will denote a generic constant whose value may be different from line to line.

\section{The existence and the uniqueness of solutions to an approximated problem}\label{sec-approximation}

Let us recall that  $D:=B(0,1)$ is the open unit ball in the  separable Hilbert space $H$. We endow it with the metric inherited from $H$.
Let us  introduce a function $\pi : H \to  \bar{D}$, called the projection onto $\bar{D}$, defined by,
for $y\in H$,  by
\begin{equation}\label{eqn-projection pi}
\pi(y)=\left \{\begin{array}{ll} y,& \mbox{if $|y|\leq 1$,}\\
\frac{y}{|y|}, & \mbox{if $|y|>1$.}\end{array}\right.
\end{equation}
Note that $\pi(y)$ and  $y-\pi(y)$ have always the same direction as  $y$. In particular,
\begin{equation}
y-\pi(y)=\lambda(\vert y\vert) y, \;\;y \in H,
\end{equation}
where the function $\lambda:[0,\infty) \to [0,1]$ is given by
\[
\lambda(r)=
\begin{cases}
  0, & \mbox{if }  r \in [0,1], \\
  1-\frac1r, & \mbox{if } r>1.
\end{cases}
\]

\begin{remark}\label{rem-projection pi}
It seems important to observe that $\pi-I$ (here $I$ stands for the identity) can be seen as the minus gradient of the function $\phi$ defined by the following formula:
\begin{equation}\label{eqn-phi}
\phi(y)=\frac12 \vert \dist(u,D)\vert^2=\left \{\begin{array}{ll} 0,& \mbox{if $|y|\leq 1$,}\\
\frac12 \bigl( \vert y\vert -1\bigr)^2, & \mbox{if $|y|>1$.}\end{array}\right.
\end{equation}
In other words,
\begin{equation}\label{eqn-phi gradient}
\pi(x)-x=-\nabla \phi (x),\;\;\; x\in H.
\end{equation}
Strictly speaking, identity \eqref{eqn-phi gradient} is not true as the function $\phi$ is not differentiable for $x \in \mathbb{S}:= \{ x \in H: \vert x \vert =1\}$.
\\
Other choices of $\phi$ are possible leading  of course to modifying the vector field $\pi$.
\end{remark}

The following Lemma states some  straightforward properties of the projection  map $\pi$ that will be used later.
\begin{lemma}\label{lem-2.1} The function $\pi$ defined in \eqref{eqn-projection pi} has the following properties.
\begin{trivlist}
\item[(i)] The map $\pi$ is globally Lipschitz in the sense that
\begin{equation}\label{e01}
|\pi (x)-\pi (y)|\leq 2|x-y|, \;\; x, y\in H.
\end{equation}
\item[(ii)] For all $\;\; x\in H$,
\begin{equation}\label{e02}
\begin{split}
(\pi (x), x-\pi (x))&= |x-\pi(x)|, \\
(x, x-\pi (x))&= \vert x \vert |x-\pi(x)|.
\end{split}
\end{equation}
\item[(iii)] For $x\in H$ and $y\in \bar{D}$,
\begin{equation}\label{e03}
(x-y, x-\pi (x))\geq 0.
\end{equation}
\end{trivlist}
\end{lemma}

\vskip 0.4cm

For every $n \in \mathbb{N}$, we consider the following penalized stochastic evolution equation:
\begin{equation}\label{eqn-2.1}
\begin{aligned}
u^n(t)&=u_0^n-\int_0^tAu^n(s)\,ds+\int_0^t\sigma(u^n(s))\,dW(s)\\
&+\int_0^tf(u^n(s))\,ds+\int_0^tB(u^n(s),u^n(s))\,ds \nonumber \\
 &-n\int_0^t(u^n(s)-\pi(u^n(s)))\,ds.\nonumber
\end{aligned}\end{equation}
Using the function $\phi$ from Remark \ref{rem-projection pi},  the above equation can be written in the following differential form
\begin{align}\label{eqn-2.1'}
du^n(t)&+Au^n(t)\,dt=\sigma(u^n(t))\,dW(t)\\
&+ \bigl[ f(u^n(t))\,+B(u^n(t),u^n(t))\, -n \nabla\phi  (u^n(s))\bigr] \,dt;\nonumber
\\
u^n(0)&=u_0^n.
\end{align}

 There exists a unique  solution $u^n$  in the framework of the Gelfand triple \eqref{eqn-Gelfand triple}. This can be proved along the same lines as the proof of the existence and the uniqueness of solutions of stochastic Navier-Stokes equations.
For the sake of completeness let us formulate the result  whose proof can be traced back to many papers, see \cite[Theorem 5.1 and Corollary 7.7]{Brz+Mot_2013} and \cite{M+Ch_2010} and references therein.
\begin{theorem} \label{thm-existence approximation}
Let us assume that  assumptions \textbf{(A.1)}-\textbf{(A.3)} are satisfied. Then there exists a unique solution
$u^n $ of problem (\ref{eqn-2.1}) such that  for $\mathbb{P}$-almost all $\omega \in \Omega $ the trajectory $u^n(\cdot , \omega )$ is equal almost everywhere to a continuous $H$-valued function defined on $[0,\infty)$
and, for every $T>0$,
\begin{equation}
\label{eqn-energy inequality}
  \mathbb{E} \Bigl[ \sup_{t \in [0,T]} |u^n(t){|}_{H}^{2} + \int_{0}^{T}\norm{u^n(t)}{}{2} \, dt  \Bigr]
  < \infty .
\end{equation}
Moreover,  this unique solution satisfies the following inequalities.
For all $p\geq 2$, $T>0$ 
 and $n\in \mathbb{N}$,
there exist   positive constants ${C}_{n}(p,T)$ and ${\bar{C}}_{n}(p,T)$  such that
\begin{equation} \label{E:H_estimate}
 \mathbb{E} \bigl( \sup_{0 \le s \le T } |u^{n} (s){|}_{H}^{p} \bigr) \le {C}_{n}(p,T) .
\end{equation}
and
\begin{equation} \label{E:HV_estimate}
  \mathbb{E} \bigl[ \int_{0}^{T} |u^{n} (s){|}_{H}^{p-2} \norm{ u^{n} (s)}{}{2} \, ds \bigr] \le {\bar{C}}_{n}(p,T)  .
\end{equation}
Finally, this unique solution satisfies, $\mathbb{P}$-almost surely,  the following inequality,
\begin{equation}
\label{eqn-monotonicity}
\lb u^n(t), u^n(t)-\pi(u^n(t))\rb \geq 0, \;\; t \geq 0.
\end{equation}
\end{theorem}
Let us observe that property \eqref{eqn-monotonicity} follows from \eqref{e03} with $y=0$.

In Lemma \ref{lem-2.2} we will strengthen the assertion \eqref{E:H_estimate} from Theorem \ref{thm-existence approximation} by  showing,  in particular, that for   every $T>0$,  $\sup_n C_{n}(4,T)<\infty$. For that aim we will use the fact that $C_{n}(8,T)<\infty$ for every $n\in\mathbb{N}$.

In the next section we will show that the limit of $u^n$, as $n\to \infty$,  exists and  is a solution of SEE with reflection $(\ref{original
equation})$.

We begin with  a number of estimates for the sequence $(u^n)_{n\geq 1}$.

\begin{lemma}\label{lem-2.2}
For every $T>0$ there exist  positive constants $K_0=K_0(T)$ and $K_1=K_1(T)$ such that the following estimates hold:
\begin{equation}\label{eqn-2.2}
\sup_n\mathbb{E}\,[\sup_{t \in [0,T]}|u^n(t)|_H^4]\leq K_0,
\end{equation}
\begin{equation}\label{2.3}
\mathbb{E}\bigl[\int_0^T|u^n(t)|_H^2 \lb u^n(t),u^n(t)-\pi(u^n(t))\rb \, dt \bigr] \leq \frac{K_1}{n}, \;\; n\in \mathbb{N}.
\end{equation}
\end{lemma}
\begin{remark}\label{lem-lem-2.2} Let us recall that  $\pi(u^n(t))$ and $u^n(t)-\pi(u^n(t))$ have the same directions. Thus,  $\mathbb{P}$-almost surely, for all $ t \geq 0$,
\begin{align*}
\lb u^n(t),u^n(t)-\pi(u^n(t))\rb&=\lb u^n(t)-\pi(u^n(t)),u^n(t)-\pi(u^n(t))\rb\\ &+ \lb \pi(u^n(t)),u^n(t)-\pi(u^n(t))\rb
\geq \vert u^n(t)-\pi(u^n(t) \vert^2_H.
\end{align*}
Therefore, the \textit{a priori} estimate \eqref{2.3} implies the following inequality which will be used later in the proof of  Lemma \ref{lem-2.4}. There exists a constant $K_1=K_1(T)>0$ such that
\begin{equation}\label{eqn-2.3'}
\mathbb{E}\bigl[\int_0^T|u^n(t)|_H^2 \vert u^n(t)-\pi(u^n(t)) \vert^2_H  \, dt \bigr] \leq \frac{K_1}{n}, \;\; n\in \mathbb{N}.
\end{equation}
\end{remark}

\vskip 0.3cm

The proof of inequality \eqref{eqn-2.2} is similar to the proof of the  estimates (5.4) in Lemma 5.3  from   \cite[Appendix A]{Brz+Mot_2013}.

\begin{proof}[Proof of Lemma \ref{lem-2.2}] Let us choose and fix $T>0$.  Let us fix $n\in \mathbb{N}$ and $u_0 \in H$. Let $\psi(z)=|z|_H^4, z\in H$.  It is easy to see that $\psi$  is of $C^2$ class and
	\begin{align*}
\psi^{\prime}(z)=4|z|_H^2z \mbox{ and }\psi^{\prime\prime}(z)=8z\otimes z+4|z|_H^2I_H, \;\; z\in H,
	\end{align*} where $I_H$ stands for the identity operator on $H$.  Applying  the It\^o formula from \cite{Pardoux79} and assumption (\ref{1.2}) we have
		\begin{equation}\label{2.4}\begin{aligned}
|u^n(t)|_H^4&=|u^n(0)|_H^4-4\int_0^t |u^n(s)|_H^2\lb u^n(s), Au^n(s)\rb\, ds\nonumber\\
  &+4 \int_0^t |u^n(s)|_H^2\lb u^n(s), \sigma(u^n(s))\rb dW(s)+4\int_0^t |u^n(s)|_H^2\lb u^n(s), f(u^n(s))\rb \,ds\nonumber\\
  &-4n\int_0^t |u^n(s)|_H^2 \,\lb u^n(s), u^n(s)-\pi(u^n(s))\rb \,ds+4 \int_0^t \lb u^n(s),\sigma(u^n(s))\rb ^2\,ds\nonumber\\
  &+2 \int_0^t |u^n(s)|_H^2\lb \sigma(u^n(s)), \sigma(u^n(s))\rb \,ds.
  \end{aligned}\end{equation}

In the first instance, we do not know whether the process $M$ defined by

\[M(t):=\int_0^t |u^n(s)|_H^2\lb u^n(s), \sigma(u^n(s))\rb dW(s), \;\; t\geq 0,\] is a martingale.
Fortunately, this is the case as it follows from inequality \eqref{E:HV_estimate}.

\dela{
Thus, as in for instance in the proof of Lemma 7.3  \cite{Brz+Mot_2013} we will use a sequence of stopping times as follows:
\begin{equation}  \label{eqn-stopping time}
   {\tau }_{N} := \inf \{ t \in [0,T] : \int_0^t |u^n(s){|}_{H}^{8} \, ds  >N  \} ,
\qquad N \in \mathbb{N} .
\end{equation}
Let us recall, that by definition $\inf \emptyset =T$. Let us also observe that in view of inequality \eqref{E:HV_estimate}, $\tau_N>0$ with $\mathbb{P}=1$.
}

  Observe that by \eqref{eqn-monotonicity} we have $\lb u^n(s), u^n(s)-\pi(u^n(s))\rb \geq 0$ for all $s\in [0,T]$.   We also note  that
  \begin{equation}4\int_0^t |u^n(s)|_H^2\lb u^n(s), Au^n(s)\rb\, ds \geq 0.\end{equation}
Since the function $f: H\rightarrow V^*$ is of linear growth, there exists a constant $C$ such that
\begin{equation}\label{f-estimate}
|\lb f(u), u \rb|\leq C\Vert u\Vert (1+|u|_H), \quad\quad u\in V.
\end{equation}
By the Young  inequality and inequality (\ref{f-estimate}) we have
\begin{equation}\label{f-estimate-1}\begin{aligned}
&4\int_0^t |u^n(s)|_H^2\lb u^n(s), f(u^n(s))\rb \,ds\nonumber\\
&\leq 2\int_0^t \Vert u^n(s)\Vert^2|u^n(s)|_H^2 \,ds +C\int_0^t (|u^n(s)|_H(1+|u^n(s)|_H))^2 \,ds\nonumber\\
&=2\int_0^t |u^n(s)|_H^2\lb u^n(s), Au^n(s)\rb \,ds +C\int_0^t (|u^n(s)|_H(1+|u^n(s)|_H))^2 \,ds
\end{aligned}\end{equation}

  Rearranging terms and taking the expectation in (\ref{2.4}) yield
\begin{align}\label{2.5}
&\hspace{-1.5truecm}\mathbb{E}\,[\sup_{0\leq r\leq t}|u^n(r)|_H^4]+4n\mathbb{E}\,[\int_0^t |u^n(s)|_H^2\lb u^n(s), u^n(s)-\pi(u^n(s))\rb \,ds]\nonumber\\
&\leq |u^n(0)|_H^4+4\mathbb{E}\,[\sup_{0\leq r\leq t}|\int_0^r |u^n(s)|_H^2\lb u^n(s), (u^n(s))\rb dW(s)|]\nonumber\\
  &+4\mathbb{E}\,[\int_0^t|u^n(s)|_H^2|\lb u^n(s), f(u^n(s))\rb |ds]+4\mathbb{E}\,[\int_0^t \lb u^n(s),\sigma(u^n(s))\rb ^2\,ds]\nonumber\\
  &+2\mathbb{E}\,[\int_0^t |u^n(s)|_H^2\lb \sigma(u^n(s)), \sigma(u^n(s))\rb \,ds].
  \end{align}
  By the Burkholder inequality,
  \begin{align}\label{2.6}
  &\hspace{-1.5truecm} \mathbb{E}\,[\sup_{0\leq r\leq t}|\int_0^r |u^n(s)|_H^2\lb u^n(s), \sigma(u^n(s))\rb dW(s)|]\nonumber\\
  &\leq  C_1\mathbb{E}\,[(\int_0^t |u^n(s)|_H^4\lb u^n(s), \sigma(u^n(s))\rb ^2\,ds)^{\frac{1}{2}}]\nonumber\\
  &\leq C_1\mathbb{E}\,[(\sup_{0\leq r\leq t}|u^n(r)|_H^2) (\int_0^t \lb u^n(s), \sigma(u^n(s))\rb ^2\,ds)^{\frac{1}{2}}]\nonumber\\
  &\leq  \frac{1}{2}\mathbb{E}\,[(\sup_{0\leq r\leq t}|u^n(r)|_H^4)]+C_2\mathbb{E}\,[\int_0^t \lb u^n(s), \sigma(u^n(s))\rb ^2\,ds].
  \end{align}
  By the linear growth of $f$, $\sigma$, and substituting (\ref{2.6}) into (\ref{2.5}) we obtain that
  \begin{align}\label{2.7}
&\hspace{-1.5truecm}\mathbb{E}\,[\sup_{0\leq r\leq t}|u^n(r)|_H^4]+4n\mathbb{E}\,[\int_0^t |u^n(s)|_H^2\lb u^n(s), u^n(s)-\pi(u^n(s))\rb \,ds]\nonumber\\
&\leq C|u^n(0)|_H^4+C\mathbb{E}\,[\int_0^t (1+|u^n(s)|_H^4)\,ds].
  \end{align}
By applying the Gronwall Lemma the above inequality   implies inequality (\ref{eqn-2.2}).  Finally, combination of inequalities \eqref{2.7} \eqref{eqn-2.2} implies inequality (\ref{2.3}).
Hence, the proof of Lemma \ref{lem-2.2} is complete.
  \end{proof}

\begin{lemma}\label{lem-2.3}
For every $T>0$ there exist constants $M_1=M_1(T)$ and $M_2=M_2(T)$ such that
\begin{equation}\label{ee2.2}
\sup_n\mathbb{E}\,[\bigg (n\int_0^T|u^n(s)-\pi(u^n(s))|_H\,ds\bigg )^2\,]\leq M_1,
\end{equation}
and
\begin{equation}\label{ee2.2-1}
\sup_n\mathbb{E}\,[\int_0^T \Vert u^n(s) \Vert ^2\,ds]\leq M_2.
\end{equation}
\end{lemma}
\vskip 0.3cm
\noindent \begin{proof}[Proof of Lemma \ref{lem-2.3}] Let us choose and fix $T>0$.  By the It\^o  formula, see  \cite{Pardoux79},  we have
\begin{equation}\label{ee2.4}\begin{aligned}
|u^n(t)|_H^2&=|u(0)|_H^2-2\int_0^t\lb u^n(s), Au^n(s)\rb \,ds\nonumber\\
  &+2\int_0^t\lb u^n(s), \sigma(u^n(s))\rb dW(s)+2\int_0^t \lb u^n(s), f(u^n(s))\rb \,ds\nonumber\\
  &-2n\int_0^t\lb u^n(s), u^n(s)-\pi(u^n(s))\rb \,ds+\int_0^t |\sigma(u^n(s))|_H^2\,ds.
  \end{aligned}\end{equation}
  Note that $\lb u^n(s), Au^n(s)\rb = \Vert u^n(s) \Vert^2$, $s \in [0,T]$ and
  \begin{equation}|\sigma(u^n(s))|_H\leq C(1+|u^n(s)|_H),\;\; s \in [0,T].\end{equation}
By the Young  inequality and inequality (\ref{f-estimate}) we have
\begin{equation}\label{f-estimate-2}
\begin{aligned}
2\int_0^t \lb u^n(s), f(u^n(s))\rb \,ds
&\leq \int_0^t \Vert u^n(s)\Vert^2ds +C\int_0^t (1+|u^n(s)|_H)^2 \,ds\\
&=\int_0^t \lb u^n(s), Au^n(s)\rb \,ds +C\int_0^t (1+|u^n(s)|_H)^2 \,ds.
\nonumber
\end{aligned}\end{equation}
In view of property (ii) in Lemma \ref{lem-2.1},
\begin{align}\label{ee2.5}
&\hspace{-1truecm}2n\int_0^t\lb u^n(s), u^n(s)-\pi(u^n(s))\rb \,ds\nonumber\\
&=2n\int_0^t|u^n(s)-\pi(u^n(s))|_H^2\,ds+2n\int_0^t\lb \pi(u^n(s)), u^n(s)-\pi(u^n(s))\rb \,ds\nonumber\\
&= 2n\int_0^t|u^n(s)-\pi(u^n(s))|_H^2\,ds+ 2n\int_0^t|u^n(s)-\pi(u^n(s))|_H\,ds,\;\; t \in [0,T].
  \end{align}
By Lemma \ref{lem-2.2}\dela{2.2}, inequalities  (\ref{ee2.4}), (\ref{f-estimate-2}),  (\ref{ee2.5}) and the Burkholder inequality we arrive at
\begin{equation}\label{ee2.6}
\sup_n\mathbb{E}\,[\bigg (n\int_0^T|u^n(s)-\pi(u^n(s))|_{H}\,ds\bigg )^2\,]
\leq  C+C\sup_n\mathbb{E}\,[\sup_{t \in [0,T]}|u^n(t)|_H^4]\leq M_1(T)
\end{equation}
and
\begin{equation}\label{ee2.6-1}
\sup_n\mathbb{E}\,[\int_0^T \Vert u^n(s) \Vert ^2\,ds]\leq M_2(T).
\end{equation}
Hence the proof of Lemma \ref{lem-2.3} is complete.
\end{proof}

In the proof of the next Lemma \ref{lem-2.4}  we will use two auxiliary functions. Let us list them for the convenience of the reader in a separate item.

\begin{lemma}\label{lem-auxiliary}
Let us define two functions $G,g:H \to [0,\infty)$ by, for $y\in H$,
\begin{align*}
G(y)&=d(y,\bar{D})^4,\;\;\ g(y)=d(y,\bar{D})^2.
\end{align*} Then  for all $y,v,h\in H$, the following identities hold
\begin{align}\label{eqn-g-G}
\nabla G(y)&=4g(y)(y-\pi (y))\nonumber\\
    g(y)\vert y-\pi(y)\vert^2 &=  G(y), \;\; y \in  H,\\
    g(y) &\leq  |y-\pi(y)|_H^2, \;\; y \in  H,
    \label{eqn-g-G-2}\\
    \label{2.9}
G^{\prime}(y)(v)&=4g(y)\lb y-\pi(y), v\rb ,
\\
\label{2.10}
   G^{\prime\prime}(y)(h,v)&=8\lb y-\pi(y), h\rb \lb y-\pi(y), v\rb \\
  &+4g(y)1_{|y|_H>1}\left[ \lb h,v\rb (1-\frac{1}{|y|_H})+\frac{1}{|y|_H^3}\lb y,h\rb \lb y,v\rb \right].
  \nonumber
\end{align}
  \end{lemma}
\begin{proof}This Lemma
follows by  simple calculations by using the following formula:
\begin{equation}\label{eqn-g}
g(y)=\vert y -\pi(y)\vert^2=\left \{\begin{array}{ll} 0,& \mbox{if $y\in H, \; |y|< 1$,}\\
 \vert y -\frac{y}{\vert y \vert} \vert^2=\bigl( \vert y\vert -1\bigr)^2, & \mbox{if $y\in H, \; |y| \geq 1$.}\end{array}\right.
\end{equation}

\end{proof}

\begin{lemma}\label{lem-2.4}
In the above framework,  the following holds for every $T>0$
\begin{equation}\label{2.8}
\lim_{n\rightarrow \infty}\mathbb{E}\,[\sup_{t \in [0,T]}|u^n(t)-\pi(u^n(t))|_H^4\,]=0.
\end{equation}
\end{lemma}
\vskip 0.3cm
\begin{proof}[Proof of Lemma \ref{lem-2.4}]   The proof of this Lemma is somehow similar to a proof of the Burkholder inequality for stochastic evolution equations, see for instance  \cite{BP-max},

Let us choose and fix $T>0$.  Let $G$ and $g$ be the functions from Lemma \ref{lem-auxiliary}.
Applying  the It\^o  formula, see \cite{Pardoux79}, since $G(u^n(0))= 0$ we have
\begin{align}\label{2.11}
G(u^n(t))&=-4\int_0^t g(u^n(s))\lb u^n(s)-\pi(u^n(s)), Au^n(s)\rb \, ds \nonumber\\
  &+4 \int_0^tg(u^n(s))\lb u^n(s)-\pi(u^n(s)), \sigma(u^n(s))\rb dW(s)\nonumber\\
&+4\int_0^tg(u^n(s))\lb u^n(s)-\pi(u^n(s)), f(u^n(s))\rb \,ds \nonumber\\
&+4\int_0^tg(u^n(s))\lb u^n(s)-\pi(u^n(s)), B(u^n(s),u^n(s))\rb \,ds \nonumber\\
  &-4n \int_0^tg(u^n(s))|u^n(s)-\pi(u^n(s))|_H^2\,ds\nonumber\\
&+4 \int_0^t \lb u^n(s)-\pi(u^n(s)),\sigma(u^n(s))\rb ^2\,ds\nonumber\\
  &+2\int_0^tg(u^n(s))1_{|u^n(s)|>1}\Bigl[|\sigma(u^n(s))|_H^2 (1-\frac{1}{|u^n(s)|_H})\nonumber \\
  & \hspace{5truecm}+\frac{1}{|u^n(s)|_H^3}\lb u^n(s),\sigma(u^n(s))\rb ^2\Bigr]\, ds   \nonumber\\
  &:= I^n_1(t)+I^n_2(t)+I^n_3(t)+I^n_4(t)+I^n_5(t)+I^n_6(t)+I^n_7(t),\;\;t\in [0,T].
  \end{align}
  We now look at each of the seven terms separately. Clearly $I_5^n(t)\leq  0$, for $t\in [0,T]$. For the term $I_1^n$ we have, for $t\in [0,T]$,
  \begin{equation}\label{2.12}
  \begin{aligned}
-\frac{1}{4}I_1^n(t)&=\int_0^t g(u^n(s))\lb u^n(s)-\pi(u^n(s)), Au^n(s)\rb\, ds \\
&= &\int_0^t g(u^n(s))\lambda(|u^n(s)|_H)\lb u^n(s), Au^n(s)\rb \,ds
\nonumber
\\
&=\int_0^t g(u^n(s))\lambda(|u^n(s)|_H)\Vert  u^n(s) \Vert^2\, ds
\geq 0.
\nonumber  \end{aligned}\end{equation}
For the term $I^n_4$, by assumption (\ref{1.2}) we have, for $t\in [0,T]$,
\begin{equation}\label{2.12-1}
\begin{aligned}
I_4^n(t)&=4\int_0^t g(u^n(s))\lb u^n(s)-\pi(u^n(s)), B(u^n(s),u^n(s))\rb \,ds \nonumber\\
&= &4\int_0^t g(u^n(s))\lambda(|u^n(s)|_H)\lb u^n(s), B(u^n(s),u^n(s))\rb \,ds= 0.
  \end{aligned}\end{equation}
  By the Burkholder inequality and inequality \eqref{eqn-g-G} we infer that for $t\in [0,T]$,
\begin{equation}\label{2.13}
\begin{aligned}
&\hspace{-1.5truecm}\mathbb{E}\,[\sup_{0\leq s\leq t}|I_2^n(s)|]
\leq C \mathbb{E}\,[(\int_0^tg(u^n(s))^2\lb u^n(s)-\pi(u^n(s)), \sigma(u^n(s))\rb ^2\,ds)^{\frac{1}{2}}]\\
&\leq C \mathbb{E}\,[(\int_0^tg(u^n(s)) \vert  u^n(s)-\pi(u^n(s))\vert^2  g(u^n(s)) \vert \sigma(u^n(s))\vert^2\,ds)^{\frac{1}{2}}]\nonumber\\
&\leq C\mathbb{E}\,[\sup_{0\leq s\leq t}(g(u^n(s)))^{\frac{1}{2}}(\int_0^t  \lb u^n(s)-\pi(u^n(s)),  \sigma(u^n(s))\rb ^2g(u^n(s)) \vert \sigma(u^n(s))\vert^2 \,ds)^{\frac{1}{2}}]\nonumber\\
&\leq \frac{1}{2}\mathbb{E}\,[\sup_{0\leq s\leq t}g(u^n(s))] +  C_{1}\mathbb{E}\,[\int_0^t\lb u^n(s)-\pi(u^n(s)), \sigma(u^n(s))\rb ^2g(u^n(s)) \vert \sigma(u^n(s))\vert^2\,ds]\nonumber\\
&\leq \frac{1}{2}\mathbb{E}\,[\sup_{0\leq s\leq t}g(u^n(s))]+C_2\mathbb{E}\,[\int_0^t(1+ |u^n(s)|_H^2) |u^n(s)-\pi(u^n(s))|_H^2\,ds].
\nonumber
\end{aligned}\end{equation}
By the Young  inequality and inequality (\ref{f-estimate}) we infer that
\begin{equation}\label{f-estimate-3}
\begin{aligned}
I_3^n&=4\int_0^tg(u^n(s))\lambda(|u^n(s)|_H)\lb u^n(s), f(u^n(s))\rb \,ds\nonumber\\
&\leq 2\int_0^tg(u^n(s))\lambda(|u^n(s)|_H) \Vert u^n(s)\Vert^2ds +C\int_0^t g(u^n(s))\lambda(|u^n(s)|_H)(1+|u^n(s)|_H)^2 \,ds\nonumber\\
&\leq &2\int_0^t g(u^n(s))\lb u^n(s)-\pi(u^n(s)), Au^n(s)\rb \,ds +C\int_0^tg(u^n(s)) (1+|u^n(s)|_H)^2 \,ds.\nonumber\\
&
\end{aligned}\end{equation}
Let us now observe that  by  \eqref{eqn-g-G-2}   we deduce  that
\begin{equation}\label{2.13-1}
I_3^n(t)
\leq 2\int_0^t g(u^n(s))\lb u^n(s)-\pi(u^n(s)), Au^n(s)\rb \,ds+ C\int_0^t \, |u^n(s)-\pi(u^n(s))|_H^2(1+|u^n(s)|_H^2)ds.
\end{equation}

Using the linear growth of  $\sigma$ we similarly have
\begin{equation}\label{2.13-2}
\begin{aligned}
&I_6^n(t)\leq  \int_0^t|u^n(s)-\pi(u^n(s))|_H^2(1+|u^n(s)|_H^2) ds,
\end{aligned}\end{equation}
and
\begin{equation}\label{2.13-3}
I_7^n(t)\leq  \int_0^t|u^n(s)-\pi(u^n(s))|_H^2(1+|u^n(s)|_H^2) ds.
\end{equation}
Hence, from some of the above estimates we infer that, for $t\in [0,T]$,

\begin{align}
\sum_{j\not=2}I^n_j(t) & \leq C\int_0^t \, |u^n(s)-\pi(u^n(s))|_H^2(1+|u^n(s)|_H^2)ds.
  \end{align}
Combining the above inequality with inequality   \eqref{2.13} we deduce that ,  for $t\in [0,T]$,
\begin{align*}
\mathbb{E}\,[\sup_{0\leq s\leq t}G(u^n(s))]&\leq \frac{1}{2}\mathbb{E}\,[\sup_{0\leq s\leq t}G(u^n(s))]+C_3\mathbb{E}\,[\int_0^t(1+ |u^n(s)|_H^2)  |u^n(s)-\pi(u^n(s))|_H^2\,ds].
  \end{align*}
Thus, by exploiting the constant $\frac12$  on the RHS above we infer that
\begin{align*}
\mathbb{E}\,[\sup_{0\leq s\leq t}G(u^n(s))]&\leq 2C_3\mathbb{E}\,[\int_0^t(1+ |u^n(s)|_H^2) |u^n(s)-\pi(u^n(s))|_H^2\,ds].
  \end{align*}

Thus, by  \eqref{eqn-2.3'}, we infer that  as $n\rightarrow \infty$,
\begin{equation}\label{2.14'}
\lim_{n \to \infty} \mathbb{E}\,[\sup_{t \in [0,T]}|u^n(t)-\pi(u^n(t))|_{H}^4\,]=\lim_{n \to \infty}\mathbb{E}\,[\sup_{t \in [0,T]}G(u^n(t))]=0.
\end{equation}
This concludes the proof of Lemma \ref{lem-2.4}.
    \end{proof}

\vskip 0.3cm

\section{The existence and the uniqueness of solutions to the reflected problem}\label{sec-main}

The aim of this section is to formulate  and prove the following main result of the paper.

\begin{theorem}\label{thm-main}
Let us assume that  assumptions \textbf{(A.1)}-\textbf{(A.3)} are satisfied. The reflected stochastic evolution equation (\ref{original equation})  admits a unique solution $(u, L)$ that satisfies, for $T> 0$,
\begin{equation}\label{2.15}
\mathbb{E}\,\bigl[\sup_{t \in [0,T]}|u(t)|_H^2\, +\, \int_0^T \Vert u(t) \Vert^2\,dt\,\bigr]< \infty.
\end{equation}
\end{theorem}
\vskip 0.3cm

\begin{remark}
\label{rem-bv} Let us recall, see \cite[\S 17 and Theorem III.2.1, p. 358]{Dinculeanu_1967}, that if a function $L:[0,\infty)\to H$ is of locally bounded variation,  then there exists a unique function  $\mu: \bigcup _{t>0}\mathcal{B}([0,t]) \to H$ such  that $\mu((s,t])=L(t+)-L(s+)$.
This measure is denoted by $dL(t)$.

\end{remark}

\begin{proof}[Proof of Theorem \ref{thm-main}]
Recall that  $u_0\in \bar{D}$.

We will show that the sequence $\{ u^n, n\geq 1\}$ defined in (\ref{eqn-2.1}) converges to a solution to equation (\ref{original equation}).

Let us observe that without a loss of generality we can fix $T> 0$ for the remainder of the proof.

For $\lambda> 0$ and $n\in\mathbb{N}$ let us define a process $f_n$ by the following formula.
\begin{equation}f_n(t)=\exp(-\lambda \int_0^t \Vert u^n(s) \Vert ^2\,ds),\;\; t\geq 0.\end{equation}

Our proof here has some common features with the proof of the uniqueness to the 2D stochastic Navier-Stokes Equations, see \cite{Schmalfuss_1997} and/or \cite{Brz+Mot_2013}.

\textbf{Step 1} We will show the following auxiliary result.
\begin{lemma}\label{lem-aux-0}
There exists an adapted process $u$ with trajectories in the space $ C([0,T], H)\cap L^2([0, T], V)$ such that
  \begin{equation}\label{2.21*}
  \lim_{n \rightarrow \infty}\{\sup_{s\in [0,T]}|u^n(s)-u(s)|_H^2+ \int_0^T  \Vert u^n(s)-u(s) \Vert ^2\,ds\}=0 \mbox{ in probability } \mathbb{P}.
  \end{equation}
\end{lemma}
\begin{proof}[Proof of Lemma \ref{lem-aux-0}]
Let us choose and fix natural numbers $m\geq n$.

Applying  the It\^o  formula from \cite{Pardoux79} we infer that
\begin{equation}\label{2.16}
\begin{aligned}
&f_n(t)|u^n(t)-u^m(t)|_H^2\nonumber\\
&=-\lambda \int_0^tf_n(s) \Vert u^n(s) \Vert ^2|u^n(t)-u^m(t)|_H^2\,ds\nonumber\\
&-2\int_0^tf_n(s)(u^n(s)-u^m(s),  A(u^n(s)-u^m(s)))\,ds\nonumber\\
  &+2 \int_0^tf_n(s)\lb u^n(s)-u^m(s), \sigma(u^n(s))-\sigma(u^m(s))\rb dW(s)\nonumber\\
  &+2\int_0^tf_n(s) \lb u^n(s)-u^m(s), f(u^n(s))-f(u^m(s))\rb \,ds\nonumber\\
 &+2\int_0^tf_n(s) \lb u^n(s)-u^m(s), B(u^n(s),u^n(s))-B(u^m(s),u^m(s))\rb \,ds\nonumber\\
  &-2n\int_0^tf_n(s)\lb u^n(s)-u^m(s), u^n(s)-\pi(u^n(s))\rb \,ds\nonumber\\
&+ 2m\int_0^tf_n(s)\lb u^n(s)-u^m(s),u^m(s)-\pi(u^m(s))\rb \,ds\nonumber\\
  &+ \int_0^t f_n(s)|\sigma(u^n(s))-\sigma(u^m(s))|_H^2\,ds\nonumber\\
  &:=&I^{n,m}_1(t)+I^{n,m}_2(t)+I^{n,m}_3(t)+I^{n,m}_4(t)+I^{n,m}_5(t)+I^{n,m}_6(t)+I^{n,m}_7(t)+I^{n,m}_8(t).
  \end{aligned}\end{equation}
Observe that
\begin{equation}\label{2.16-1}
   I^{n,m}_2(t)=-2\int_0^tf_n(s) \Vert u^n(s)-u^m(s) \Vert ^2\,ds.
   \end{equation}
By the assumption on $f$ and the Young  inequality, we have
\begin{equation}\label{f-estimate-4}
\begin{aligned}
I^{n,m}_4(t)&\leq & C \int_0^tf_n(s) \Vert u^n(s)-u^m(s) \Vert \Vert f(u^n(s))-f(u^m(s))\Vert_{V^*} ds\nonumber\\
&\leq  C \int_0^tf_n(s) \Vert u^n(s)-u^m(s) \Vert \vert u^n(s)-u^m(s)\vert_H ds\nonumber\\
&\leq \frac{1}{2} \int_0^tf_n(s) \Vert u^n(s)-u^m(s) \Vert^2 \,ds +C \int_0^tf_n(s) \vert u^n(s)-u^m(s)\vert_H^2 \,ds.
\end{aligned}\end{equation}

Now,
\begin{equation}\label{2.16-2}
   I^{n,m}_5(t)=2\int_0^tf_n(s)\left( \bar{b}(u^n(s),u^n(s),u^n(s)-u^m(s))- \bar{b}(u^m(s),u^m(s),u^n(s)-u^m(s))\right)
   \end{equation}
By assumption (\ref{1.2}), we have
\begin{align}\label{eqn-test-01}
\bar{b}(u^m(s),u^m(s),u^n(s)-u^m(s))&=\bar{b}(u^m(s),u^n(s),u^n(s)-u^m(s))
\\&\hspace{-3truecm}-\bar{b}(u^m(s),u^n(s)-u^m(s),u^n(s)-u^m(s))=\bar{b}(u^m(s),u^n(s),u^n(s)-u^m(s)).
\nonumber
\end{align}
Therefore, in view of assumption (\ref{1.2-b}) we infer  that
\begin{equation}\label{2.16-3}
\begin{aligned}
   &\hspace{-2truecm}\left| \bar{b}(u^n(s),u^n(s),u^n(s)-u^m(s))- \bar{b}(u^m(s),u^m(s),u^n(s)-u^m(s))\right|\nonumber\\
&=\left| \bar{b}(u^n(s)-u^m(s),u^n(s),u^n(s)-u^m(s))\right|\nonumber\\
&\leq2  \Vert u^n(s) \Vert  |u^n(s)-u^m(s)|_H  \Vert u^n(s)-u^m(s) \Vert .
   \end{aligned}\end{equation}
It follows from (\ref{2.16-2}), (\ref{2.16-3}) that

\begin{equation}\label{2.16-4}
\begin{aligned}
   I^{n,m}_5(t)&\leq &4\int_0^tf_n(s) \Vert u^n(s) \Vert  |u^n(s)-u^m(s)|_H  \Vert u^n(s)-u^m(s) \Vert \, ds\\
&\leq &\int_0^tf_n(s) \Vert u^n(s)-u^m(s) \Vert ^2\,ds+ 4\int_0^tf_n(s) \Vert u^n(s) \Vert ^2 |u^n(s)-u^m(s)|_H^2\,ds,
   \nonumber \end{aligned}\end{equation}
where the inequality $4ab\leq a^2+4b^2$ has been used.

  As $\pi(u^n(s))\in \bar{D}$ and $ \pi(u^m(s))\in \bar{D}$, it follows from \eqref{e03} in part  (iii) of Lemma \ref{lem-2.1} that $\lb u^n(s)-\pi(u^m(s)), u^n(s)-\pi(u^n(s))\rb \geq 0$ and
   $\lb u^m(s)-\pi(u^n(s)), u^m(s)-\pi(u^m(s))\rb \geq 0$. Hence,
   \begin{equation}\label{2.17}
   \begin{aligned}
   I^{n,m}_6(t)&=-2n\int_0^tf_n(s)\lb u^n(s)-\pi(u^m(s)), (u^n(s)-\pi(u^n(s))\rb \,ds\nonumber\\
   &+2n\int_0^tf_n(s)\lb u^m(s)-\pi(u^m(s)), (u^n(s)-\pi(u^n(s)))\rb \,ds\nonumber\\
   &\leq2n\int_0^tf_n(s)\lb u^m(s)-\pi(u^m(s)), (u^n(s)-\pi(u^n(s)))\rb \,ds\nonumber\\
   &\leq 2n (\int_0^t|u^n(s)-\pi (u^n(s))|_H\,ds) \sup_{0\leq s\leq t}|u^m(s)-\pi(u^m(s))|_{H},
   \end{aligned}\end{equation}
as $f_n(s)\leq 1$.
   A similar calculation also yields
   \begin{equation}\label{2.18}
   I^{n,m}_7(t)\leq (2m \int_0^t|u^m(s)-\pi (u^m(s))|_H\,ds) \sup_{0\leq s\leq t}|u^n(s)-\pi(u^n(s))|_H
   \end{equation}
Substituting (\ref{2.16-1}), \eqref{f-estimate-3},  (\ref{2.16-2}), (\ref{2.18}), (\ref{2.17}) into (\ref{2.16}), choosing $\lambda > 4$, as in the proof of Lemma \ref{lem-2.4}\dela{2.4}, using the Burkholder and the  H\"older inequalities, as well as the Lipschitz continuity of the maps  $\sigma$,  we obtain that
\begin{equation}
\label{2.19}
\begin{aligned}
&\mathbb{E}\,[\sup_{0\leq s\leq t}f_n(s)|u^n(s)-u^m(s)|_H^2\,]+ \mathbb{E}\,[\int_0^tf_n(s)  \Vert u^n(s)-u^m(s) \Vert ^2\,ds]\\
  &\leq &\frac{1}{2}\mathbb{E}\,[\sup_{0\leq s\leq t}f_n(s)|u^n(s)-u^m(s)|_H^2\,]+C\mathbb{E}\,[\int_0^tf_n(s)|u^n(s)-u^m(s)|_H^2\,ds]\nonumber\\
  &+ C(\mathbb{E}\,[\left ( 2n \int_0^t|u^n(s)-\pi (u^n(s))|_{H}ds\right )^2\,])^{\frac{1}{2}}(\mathbb{E}\,[ \sup_{0\leq s\leq t}|u^m(s)-\pi(u^m(s))|_{H}^2\,])^{\frac{1}{2}}  \nonumber\\
  &+C(\mathbb{E}\,[\left ( 2m \int_0^t|u^m(s)-\pi (u^m(s))|_{H}ds\right )^2\,])^{\frac{1}{2}}(\mathbb{E}\,[ \sup_{0\leq s\leq t}|u^n(s)-\pi(u^n(s))|_{H}^2\,])^{\frac{1}{2}}.
  \nonumber
  \end{aligned}\end{equation}
  Remark that the first term on the right side of (\ref{2.19}) comes from the estimate of $ I^{n,m}_3$ and the Young  inequality.
  By the Gronwall Lemma and also Lemma \ref{lem-2.3} \dela{2.3}, we get
  \begin{equation}\label{2.20}\begin{aligned}
&\mathbb{E}\,[\sup_{s\in [0,T]}f_n(s)|u^n(s)-u^m(s)|_H^2\,]+\mathbb{E}\,[\int_0^tf_n(s)  \Vert u^n(s)-u^m(s) \Vert ^2\,ds]\nonumber\\
  &\leq & C(M_T)^{\frac{1}{2}}(\mathbb{E}\,[ \sup_{s\in [0,T]}|u^m(s)-\pi(u^m(s))|_H^2\,])^{\frac{1}{2}}  \nonumber\\
  &+C(M_T)^{\frac{1}{2}}(\mathbb{E}\,[ \sup_{s\in [0,T]}|u^n(s)-\pi(u^n(s))|_H^2\,])^{\frac{1}{2}}.
  \end{aligned}\end{equation}
  By Lemma \ref{lem-2.4}\dela{2.4}, the above yields  that
  \begin{equation}\label{2.20-1}
\lim_{n,m\rightarrow \infty}\{\mathbb{E}\,[\sup_{s\in [0,T]}f_n(s)|u^n(s)-u^m(s)|_H^2\,]+ \mathbb{E}\,[\int_0^Tf_n(s)  \Vert u^n(s)-u^m(s) \Vert ^2\,ds]\}=0.
\end{equation}
Next we will show the following auxiliary result in which we use a classical concept,    see e.g. Definition 1.2.6 in \cite{Geis}.
\begin{lemma}\label{lem-aux-1}
The sequence $\{u^n: n\geq 1\}$ of $C([0,T], H)\cap L^2([0, T], V)$-valued random variables is   Cauchy   in probability.
\end{lemma}
\begin{proof}[Proof of Lemma \ref{lem-aux-1}]
Indeed, given $\delta> 0$, for any $M> 0$ we have
 \begin{equation}\label{2.20-2}\begin{aligned}
&\mathbb{P}(\sup_{s\in [0,T]}|u^n(s)-u^m(s)|_H^2+\int_0^T \Vert u^n(s)-u^m(s) \Vert ^2\,ds\geq \delta)\nonumber\\
&\leq \mathbb{P}(\sup_{s\in [0,T]}|u^n(s)-u^m(s)|_H^2+\int_0^T \Vert u^n(s)-u^m(s) \Vert ^2\,ds\geq \delta, \int_0^T \Vert u^n(s) \Vert ^2\,ds\leq M)\nonumber\\
&\quad\quad\quad +\mathbb{P}(\int_0^T \Vert u^n(s) \Vert ^2\,ds> M)\nonumber\\
&\leq \mathbb{P}(\sup_{s\in [0,T]}f_n(s)|u^n(s)-u^m(s)|_H^2+\int_0^Tf_n(s) \Vert u^n(s)-u^m(s) \Vert ^2\,ds\geq  e^{-\lambda TM}\delta)
\\
&+\frac{1}{M}\mathbb{E}\,[\int_0^T \Vert u^n(s) \Vert ^2\,ds]\nonumber\\
&\leq e^{\lambda TM}\frac{1}{\delta}\left\{E[\sup_{s\in [0,T]}f_n(s)|u^n(s)-u^m(s)|_H^2\,]
+\mathbb{E}\,[\int_0^tf_n(s)  \Vert u^n(s)-u^m(s) \Vert ^2\,ds]\right\}\nonumber\\
&+\frac{1}{M}\mathbb{E}\,[\int_0^T \Vert u^n(s) \Vert ^2\,ds].
\end{aligned}\end{equation}
Now, for any $\varepsilon> 0$, by Lemma \ref{lem-2.3}\dela{2.3} we can  choose $M> 0$ such that
\begin{equation}\label{2.20-3}
   \frac{1}{M}\mathbb{E}\,[\int_0^T \Vert u^n(s) \Vert ^2\,ds]\leq \varepsilon, \quad\quad\quad \mbox{for all}\quad \quad n\geq 1.
   \end{equation}
Then by letting  $n,m\rightarrow \infty$ in (\ref{2.20-2}) we obtain
\begin{equation}\limsup_{n,m\rightarrow \infty}\mathbb{P}(\sup_{s\in [0,T]}|u^n(s)-u^m(s)|_H^2+\int_0^T \Vert u^n(s)-u^m(s) \Vert ^2\,ds\geq \delta)\leq \varepsilon.\end{equation}
As $\varepsilon$ is arbitrary,  we infer that
\begin{equation}
\lim_{n,m\rightarrow \infty}\mathbb{P}(\sup_{s\in [0,T]}|u^n(s)-u^m(s)|_H^2+\int_0^T \Vert u^n(s)-u^m(s) \Vert ^2\,ds\geq \delta)=0.\end{equation}
This completes the proof of Lemma \ref{lem-aux-1}.
\end{proof}
Since the space $C([0,T], H)\cap L^2([0, T], V)$ is complete, we infer that  there exists an adapted process $u$ with trajectories in the space $ C([0,T], H)\cap L^2([0, T], V)$ such that
  \begin{equation}\label{2.21}
  \lim_{n \rightarrow \infty}\bigl[ \sup_{s\in [0,T]}|u^n(s)-u(s)|_H^2+ \int_0^T  \Vert u^n(s)-u(s) \Vert ^2\,ds\bigr]=0 \mbox{ in probability}.
  \end{equation}
Hence the proof of Lemma  \ref{lem-aux-0} is complete.
\end{proof}

\textbf{Step 2.}   Applying    the Fatou Lemma, see \cite[I.28]{Rudin_1987_RCA}, in view Lemma \ref{lem-2.3}\dela{2.3} we infer
 that the process  $u$ satisfies inequality (\ref{2.15}).  Let us stress that we do not claim that the convergence of $u_n$ to $u$ is in the space $L^2(\Omega,C([0,T], H)\cap L^2([0, T], V))$.

\textbf{Step 3.}
  Our final task is to  show that the process  $u(t), t\geq 0$ is a solution to equation (\ref{original equation}).  To make this precise we need to construct an appropriate process $L$.

  For this aim let us consider a sequence $(L^{n})_{n\in \mathbb{N}}$ of   stochastic processes defined by the following formulae
  \begin{equation}\label{2.22}
  L^{n}(t)=-n\int_0^t(u^n(s)-\pi (u^n(s)))\,ds, \;\; t \geq 0.
  \end{equation}
Let us observe that according to inequality \eqref{ee2.2} in Lemma \ref{lem-2.3}\dela{2.3},
  \begin{equation}\label{eqn-2.23}
  \sup_n\mathbb{E}\,[\Var_{H}(L^{n})([0, T])^2\,]\leq \sup_n\mathbb{E}\bigl[(n\int_0^T|u^n(t)-\pi (u^n(t))|_{H}\,dt)^2\bigr]< \infty.
  \end{equation}
  \vskip 0.2cm
  Next we will show  the following auxiliary result.

\begin{lemma}\label{lem-aux-2}
The sequence $(L^n)_{n=1}^\infty$ of $C([0, T], V^{\ast})$-valued random variables   is convergent (in  $ C([0, T], V^{\ast})$)   in probability.
  Moreover,  if a process $L$ is the limit of $(L^n)$, then
it  is a $H$-valued adapted process of bounded variation and
\begin{equation*}
  \mathbb{E}\,[\Var_{H}({L})([0,T])^2\,] < \infty.
\end{equation*}
\end{lemma}
\vskip 0.3cm
\begin{proof}[Proof of Lemma \ref{lem-aux-2}]
 Let us recall from (\ref{eqn-2.1}) that the process $L^n$ is defined by
\begin{equation}\label{2.023}\begin{aligned}
{L}^n(t)&=u^n(t)+\int_0^tA u^n(s)\,ds-\int_0^t\sigma(u^n(s))\,dW(s)\nonumber\\
&-\int_0^tf(u^n(s))\,ds -\int_0^tB(u^n(s))\,ds-u(0),\;\;t\in [0,T].
\end{aligned}\end{equation}
Let us also define  a process
$L=(L(t), \,t\in [0,T])$,  by
\begin{align}\label{2.024}
{L}(t)=u(t)&+\int_0^tAu(s)\,ds-\int_0^t\sigma(u(s))\,dW(s)\nonumber\\
&-\int_0^tf(u(s))\,ds-\int_0^tB(u(s))\,ds-u(0),\;\;t\in [0,T].
\end{align}
Let us observe that ${L}^n(0)=0$ and  ${L}(0)=0$.

Since $u$ is an adapted process with trajectories in the space $C([0,T], H)\cap L^2([0, T], V)$, one can easily show, see for instance properties (i)-(iv) listed below, that
$L$ is a $V^\ast$-valued continuous and adapted process. \\
Moreover, keeping  in mind that
\begin{trivlist}
\item[(i)] the operator $A:V \to V^\ast$ is continuous,
\item[(ii)] $B$ induces a continuous bilinear map,
\begin{align*}
[C([0,T], H)\cap L^2([0, T], V)]^2 &\ni (u,v)\mapsto \{[0,T]\ni s \mapsto  B(u(s),v(s))\in V^\ast \} \\
& \in L^2([0, T], V^{\ast})
\end{align*}
see, for instance, Lemma III.3.4 in \cite{Temam_2001},
\item[(iii)] $\sigma: H \to H$ is globally Lipschitz continuous,
\item[(iv)] $f:H \to V^\ast$ is globally Lipschitz continuous,
 \end{trivlist}
 in view of  (\ref{2.21}) we infer that
as  $n\rightarrow \infty$,   the sequence ${L}^n$  converges to  ${L}$  in  $ C([0, T], V^{\ast})$ in probability.
In particular, we use  \cite[Theorem 4.3.4]{Friedman_1975_v1} to deduce from Lemma \ref{lem-aux-0} that

  \begin{equation}\label{2.21**}
  \lim_{n \rightarrow \infty}\{\sup_{t\in [0,T]}|\int_0^t\sigma(u^n(s))\,dW(s)-\int_0^t\sigma(u(s))\,dW(s)|_H^2\}=0 \mbox{ in probability}.
  \end{equation}

It is easy to see that the mapping $|\cdot |_H: V^\ast\rightarrow R^+$ (with the convention $|h|_H=\infty$ for $h\in V^\ast\setminus H$) is lower semi-continuous. Now we will provide an elementary proof for the fact that
the mapping of the total variation norm in $H$, i.e.
the function
  \begin{equation}\label{eqn-lsc of total varaintion}
\Var_{H}(\cdot)([0, T]): C([0, T], V^{\ast}) \ni v \mapsto  \Var_{H}(v)([0, T]) \in  [0,\infty]
  \end{equation}
is lower semi-continuous. Suppose  $v_n\rightarrow v$ in $C([0, T], V^{\ast})$ as $n\rightarrow \infty$. We will show
\begin{equation}\label{eqn-2.024-1}
\Var_{H}(v)([0,T])\leq \liminf_{n \to \infty} \Var_{H}(v_n)([0,T]).
\end{equation}
Without loss of generality we assume $\liminf_{n \to \infty} \Var_{H}(v_n)([0,T])<\infty$. Taking any finite partition $0=t_1<t_2<\cdots t_m=T$ of the interval $[0, T]$, we have
\begin{equation}\label{eqn-2.024-2}\begin{aligned}
\sum_{i=1}^{m-1}|v(t_{i+1})-v(t_{i})|_{H}
&\leq\sum_{i=1}^{m-1}\liminf_{n \to \infty}|v_n(t_{i+1})-v_n(t_{i})|_{H}\nonumber\\
&= \sum_{i=1}^{m-1}\lim_{n \to \infty}\inf_{k\geq n}|v_k(t_{i+1})-v_k(t_{i})|_{H}\nonumber\\
&= \lim_{n \to \infty}\sum_{i=1}^{m-1}\inf_{k\geq n}|v_k(t_{i+1})-v_k(t_{i})|_{H}\nonumber\\
&\leq \lim_{n \to \infty}\inf_{k\geq n}\big\{\sum_{i=1}^{m-1}|v_k(t_{i+1})-v_k(t_{i})|_{H}\big\}\nonumber\\
&\leq \lim_{n \to \infty}\inf_{k\geq n}\Var_{H}(v_k)([0,T])\nonumber\\
&=\liminf_{n \to \infty} \Var_{H}(v_n)([0,T]).
\end{aligned}\end{equation}
Now take the supermum of the left side over all the finite partitions of the interval $[0, T]$ to obtain (\ref{eqn-2.024-1}).
\vskip 0.3cm
By the lower semi-continuity of the total variation norm just proved,
it follows from  \eqref{2.22}  that ${L}$ is a $H$-valued process of bounded variation and $\mathbb{P}$-almost surely
\[\Var_{H}({L})([0,T])\leq \liminf_{n \to \infty} \Var_{H}({L}^n)([0,T]) \leq \liminf_{n \to \infty} \int_0^Tn|u^n(t)-\pi (u^n(t))|_{H}\,dt.\]

Hence,
by  the Fatou Lemma and inequality (\ref{eqn-2.23}), we infer that
\begin{equation}\label{eqn-2.025}
  \mathbb{E}\,[\Var_{H}({L})([0,T])^2\,]
\leq \sup_n\mathbb{E}\,[(\int_0^Tn|u^n(t)-\pi (u^n(t))|_{H}\,dt)^2\,]< \infty.
\end{equation}
Hence the proof of Lemma \ref{lem-aux-2} is complete.
\end{proof}

\vskip 0.3cm
 To complete the proof of Theorem \ref{sec-main} it is necessary to show that $(u, L)$ is a solution to equation \eqref{eqn-SEE with reflection}. To this aim  we need to verify the remaining conditions of  Definition \ref{def-solution}.
By  the Fatou Lemma and Lemma \ref{lem-2.4}\dela{2.4} we have
  \begin{equation}\label{2.26}
  \mathbb{E}\,[\sup_{s\in [0,T]}|u(s)-\pi(u(s))|_H^2\,]
 \leq \lim_{n \rightarrow \infty}\mathbb{E}\,[\sup_{s\in [0,T]}|u^n(s)-\pi(u^n(s))|_H^2\,]=0.
  \end{equation}
  This implies  that $\mathbb{P}$-almost surely, for every $t> 0$,  $u(t)=\pi(u(t))\in \bar{D}$.

Let us choose and fix function  a $\phi\in C([0,T],\bar{D})$.
Since by property \eqref{e03},  for every
  $n\in \mathbb{N}$,   $\lb u^n(t)-\phi(t), u^n(t)-\pi (u^n(t))\rb \geq 0$, we deduce by   \eqref{2.22} that almost surely
  \begin{equation}\label{eqn-2.27}
  \int_0^T(\phi(t)-u^n(t), L^{n}(dt))= -n \int_0^T (\phi(t)-u^n(t), u^n(s)-\pi (u^n(t)))\,dt  \geq 0,
  \end{equation}
  where the expression on the  LHS is the Riemann-Stieltjes integral  of the $H$-valued function $\phi(t)-u^n(t)$ with respect to an $H$-valued bounded-variation function $L^{n}$.
This will imply  that
\begin{equation}\label{2.28}
  \int_0^T(\phi(t)-u(t), L(dt))\geq 0
  \end{equation}
provided  we can show that in probability
\begin{equation}\label{2.28-1}
  \int_0^T(\phi(t)-u(t), L(dt))=\lim_{n\rightarrow \infty}\int_0^T(\phi(t)-u^n(t), L^{n}(dt)).
  \end{equation}
Let us observe that
\begin{align}\label{2.28-2}
  &\hspace{-0.2truecm}\int_0^T(\phi(t)-u^n(t), L^{n}(dt))-\int_0^T(\phi(t)-u(t), L(dt))\\
&=\int_0^T(u(t)-u^n(t), L^{n}(dt)) + \left(\int_0^T(\phi(t)-u(t), L^n(dt))-\int_0^T(\phi(t)-u(t), L(dt))\right)\nonumber\\
&=:C_1^n+C_2^n,\;\; n\in \mathbb{N}.\nonumber
  \end{align}
In view of (\ref{2.21}) and (\ref{eqn-2.23}) we infer that in probability
\begin{equation}\label{2.28-3}
 |C_1^n|\leq  \sup_{t \in [0,T]}|u^n(t)-u(t)|_H \Var_{H}(L^n)(T)
\to 0, \quad \quad \mbox{as} \quad n\rightarrow \infty.
  \end{equation}
As we only know that $L^n\rightarrow L$ in the space $C([0, T], V^{\ast})$, it is not evident that $ C_2^n\rightarrow 0$. However, we can prove this   by employing the classical density argument.
Let us put $v=\phi-u$. Since both $\phi$ and $u$ belong to $C([0,T],H)$, so does $v$.  By the density of the space $C([0,T],V)$ in the space $C([0,T],H)$, for
every  $\varepsilon> 0$, we  can choose $v_{\varepsilon}\in C([0,T],V) $ such that $\vert v_\varepsilon -v\vert_{C([0,T],H)}=\sup_{t \in [0,T]}|v_{\varepsilon}(t)-v(t)|_H^2< \varepsilon$. Then $C_2^n$ can be bounded as
\begin{align}
 |C_2^n| &= \vert \int_0^T(v(t), L^n(dt))-\int_0^T(v(t), L(dt)) \vert
\nonumber \\ &\leq    |\int_0^T(v(t)-v_\varepsilon(t)  , L^n(dt))| + |\int_0^T(v(t)-v_\varepsilon(t)  , L(dt))|\nonumber\\
&+ |\int_0^T(v_\varepsilon(t), L^n(dt))-\int_0^T(v_\varepsilon(t), L(dt))|.
\label{2.28-4}
  \end{align}
   In view of the uniform bounds (\ref{eqn-2.23}) and \eqref{eqn-2.025} on  the total variation of $L^n$ and $L$, the expectation of the first two terms on the RHS of \eqref{2.28-4}
   can be  bounded by $C\varepsilon^{\frac{1}{2}}$, while the expectation of the third term tends to zero as $n\rightarrow \infty$. As $\varepsilon$ is arbitrary, we conclude that
\begin{equation}\label{2.28-5}
\lim_{n\rightarrow \infty}C_2^n=0.
\end{equation}
Let us emphasize that the equality \eqref{2.28-5} can be seen as a version of the so called  Helly's Second Theorem, see  \cite[Theorem 1.6.10 p. 29]{Lojasiewicz_1988}.

Putting the above estimates together we infer equality (\ref{2.28-1}). Thus  we have shown that $(u, L)$ is a solution to equation (\ref{original equation}).
 \vskip 0.3cm
\begin{proof}[Proof of the  uniqueness part of Theorem \ref{thm-main}]
Our  proof is similar to the proof of Lemma   7.3  from \cite{Brz+Mot_2013} and uses the Schmalfuss idea of
application of the It\^o formula for appropriate function, see \cite{Schmalfuss_1997}.

Let  $(u, L)$ be the  solution to the reflected SEE (\ref{original equation}) constructed above.
 Let $(v, L_1)$ be another solution to the reflected SEE (\ref{original equation}).
 Set $h(t)=e^{-4\int_0^t \Vert u(s) \Vert ^2\,ds}$, $t\geq 0$.
 By the It\^o  formulae from \cite{Pardoux79} and  for real-valued processes it follows that
  \begin{equation}\label{e2.29}\begin{aligned}
  h(t)|u(t)-v(t)|_H^2 \
 &= -4\int_0^th(s) \Vert u(s) \Vert ^2|u(s)-v(s)|_H^2 \,ds\\
&-2\int_0^th(s) \Vert u(s)-v(s) \Vert ^2 \,ds\nonumber\\
 &+2\int_0^t h(s) (u(s)-v(s), \sigma(u(s))-\sigma(v(s)))\,dW(s)\nonumber\\
  &+2\int_0^th(s) (u(s)-v(s), f(u(s))-f(v(s)))\,ds\nonumber\\
&+2\int_0^th(s) \lb u(s)-v(s), B(u(s),u(s))-B(v(s),v(s))\rb \,ds\nonumber\\
 &+\int_0^th(s) |\sigma(u(s))-\sigma(v(s))|_H^2\,ds\nonumber\\
  &+2\int_0^th(s)(u(s)-v(s), L(ds))-2\int_0^th(s)(u(s)-v(s),L_1(ds)).
  \nonumber
 \end{aligned}\end{equation}

  As $v(t), u(t)\in \bar{D}$, for all $t\geq 0$, we infer from the definition of the solution that
 \begin{equation}\int_0^th(s)(u(s)-v(s), L(ds))\leq 0, \quad \int_0^th(s)(u(s)-v(s)), L_1(ds))\geq 0\end{equation}
 Hence, it follows that
\begin{equation}\label{2.34}
\begin{aligned}
h(t)|u(t)-v(t)|_H^2
 &\leq -4\int_0^th(s) \Vert u(s) \Vert ^2|u(s)-v(s)|_H^2 \,ds\nonumber\\
&-2\int_0^th(s) \Vert u(s)-v(s) \Vert ^2 \,ds\nonumber\\
 &+2\int_0^t h(s) (u(s)-v(s), \sigma(u(s))-\sigma(v(s)))\,dW(s)\nonumber\\
  &+2\int_0^th(s) (u(s)-v(s), f(u(s))-f(v(s)))\,ds\nonumber\\
&+2\int_0^th(s) \lb u(s)-v(s), B(u(s),u(s))-B(v(s),v(s))\rb \,ds\nonumber\\
 &+\int_0^th(s) |\sigma(u(s))-\sigma(v(s))|_H^2\,ds.
 \end{aligned}\end{equation}
Now, following the same argument as in the proof of (\ref{2.19}) and (\ref{2.20}),
by  Gronwall inequality, it can be  shown that $\mathbb{E}\,[h(t)|u(t)-v(t)|_H]=0$ for all $t\geq 0$ proving part  of  Theorem \ref{thm-main}.
\end{proof}

In order to formulate our next result let us recall that $L$ is an $H$-vialed process whose trajectories are locally of bounded variation.
We introduce the following notation
\begin{equation}\label{eqn-variation function}
\vert L \vert_t:= \Var([0,t])(L), \;\; t\in [0,\infty).
\end{equation}
For each $\omega \in\Omega$, the function $\mathbb{R}_+ \ni t \mapsto \vert L \vert_t$ is increasing. Hence,  one can associate with it
 a unique measure $m_{\vert L\vert}$, called the Lebesgue-Stieltjes measure.   This random  family of Lebesgue-Stieltjes measures  is traditionally denoted by $d\vert L\vert_t$.

If $m_{L}$ denotes the unique $H$-valued Lebesgue-Stieltjes measure
  associated with the $H$-valued process $L$, see \cite[Theorem III.2.1, p. 358]{Dinculeanu_1967}, then the measure $m_{\vert L\vert}$ is equal to the variation of the
  measure $m_{L}$, see \cite[pp. 361 and 363]{Dinculeanu_1967}.

\begin{proposition}\label{prop-support}
Let $(u,L)$ be the solution to equation (\ref{original equation}). Then $\mathbb{P}$-almost surely the measure $d|L|_t$ is supported on the set
\begin{equation}\label{eqn-support}
\{t \in [0,\infty):  u(t)\in \partial D\}=\{t \in [0,\infty): |u(t)|=1\}.
\end{equation}

\end{proposition}
\vskip 0.3cm
\begin{proof}[Proof of Proposition \ref{prop-support}]
The assertion of the proposition  is equivalent to 
$\mathbb{P}$-almost surely
\begin{equation}\label{2.35}
\int_0^T1_{[0,1)}(|u(t)|_H)\,d|L|_t=0,
\end{equation}
where on the LHS we have the Lebesgue-Stieltjes integral.

To this end, by an approximation argument it suffices to prove that for every $C^1$ function with compact support $\psi: [0,1) \to [0,\infty)$, the following equality holds
\begin{equation}\label{2.36}
\int_0^T\psi(|u(t)|_H)\,d|L|_t=0, \quad a.s.,
\end{equation}
or equivalently,
\begin{equation}\label{2.36-0}
E\left [\int_0^T\psi(|u(t)|_H)\,d|L|_t\right ]=0.
\end{equation}
Let us observe that since the function $\psi(|u(\cdot)|_H)$ is continuous, the  Lebesgue-Stieltjes integral on the LHS of  \eqref{2.36} is equal to the
Riemann-Stieltjes integral.


Recall that
\begin{equation}L^n_t=-n\int_0^t(u^n(s)-\pi(u^n(s)))\,ds, \;\; t \in [0,T].\end{equation}
As $L^n\rightarrow L$ in $C([0, T], V^\ast)$ in probability, by the lower semi-continuity of the function $\Var_{H}(\cdot)((s, t])$ defined  in \eqref{eqn-lsc of total varaintion} (with $[0, T]$ replaced by $(s, t]$)  we conclude that
\begin{equation}\label{2.36-1}
\Var_{H}(L)((s,t])\leq \liminf_{n\rightarrow \infty}\Var_{H}(L^n)((s,t]), \mbox{ for any $s< t$}.
\end{equation}
This implies that for any non-negative, continuous function $\eta(\cdot ): [0, T]\rightarrow [0, \infty)$, it holds that
\begin{equation}\label{2.36-2}
\int_0^T\eta(t)\,d | L|_t\leq \liminf_{n\rightarrow \infty}\int_0^T\eta(t)\,d | L^n|_t.
\end{equation}
Indeed, since $\eta$ is uniformly continuous on $[0, T]$, for any given $\varepsilon>0$ there exist a partition $0=t_0<t_1<\cdots <t_m=T$ and a simple function $\eta_m(t)=\sum_{k=0}^{m-1}a_k\chi_{(t_k, t_{k+1}]}(t)$ such that $\eta_m(t)\leq \eta(t), t\in [0, T]$, and
$\sup_{0\leq t\leq T}|\eta_m(t)-\eta(t)|\leq \varepsilon$. For example, take $a_k=\min_{t\in [t_k, t_{k+1}]}\eta(t)$. For the simple function $\eta_m$,
using the Riemann sum it follows from (\ref{2.36-1}) that
\begin{equation}\label{2.36-3}
\int_0^T\eta_m(t)\,d | L|_t\leq \liminf_{n\rightarrow \infty}\int_0^T\eta_m(t)\,d | L^n|_t.
\end{equation}
Therefore we have
\begin{equation}\label{2.36-4}
\begin{aligned}
&\int_0^T\eta(t)\,d | L|_t\leq \int_0^T\eta_m(t)\,d | L|_t+|L|_T\varepsilon \nonumber\\
&\leq \liminf_{n\rightarrow \infty}\int_0^T\eta_m(t)\,d | L^n|_t+|L|_T\varepsilon\nonumber\\
&\leq \liminf_{n\rightarrow \infty}\int_0^T\eta(t)\,d | L^n|_t+|L|_T\varepsilon,
\end{aligned}\end{equation}
here we have used the fact that $\eta_m(t)\leq \eta(t), t\in [0, T]$.
Since $\varepsilon$ is arbitrary, (\ref{2.36-2}) is proved.
In particular, letting $\eta(t)=\psi(|u(t)|_H)$ and using the Fatou Lemma, we obtain
\begin{equation}\label{2.37}
E\left[\int_0^T\psi(|u(t)|_H)\,d | L|_t\right] \leq \liminf_{n\rightarrow \infty}E\left[\int_0^T\psi(|u(t)|_H)\,d | L^n|_t\right].
\end{equation}
On the other hand, recalling $C=\sup_nE[(|L^n|_T)^2]< \infty$, by the dominated convergence theorem,   we infer that
\begin{align}\label{2.38}
\begin{split}
&E\left[\left|\int_0^T\psi(|u(t)|_H)\,d | L^n|_t-\int_0^T\psi(|u^n(t)|_H)\,d | L^n|_t\right|\right] \nonumber\\
&\leq E\left[\sup_{t \in [0,T]}|\psi(|u(t)|_H)-\psi(|u^n(t)|_H)||L^n|_T\right]
\\
&\leq C^{\frac{1}{2}} \left(E\left[\sup_{t \in [0,T]}|\psi(|u(t)|_H)-\psi(|u^n(t)|_H)|^2\right]\right)^{\frac{1}{2}}
\rightarrow 0.
\end{split}
\end{align}
 Together with  (\ref{2.37}) we deduce that
\begin{equation*}
\begin{split}
E\left[\int_0^T\psi(|u(t)|_H)\,d | L|_t\right] &\leq \liminf_{n\rightarrow \infty}E\left[\int_0^T\psi(|u^n(t)|_H)\,d | L^n|_t\right]
\\ &=\liminf_{n\rightarrow \infty}E\left[n\int_0^T\psi(|u^n(s)|_H)|u^n(s)-\pi(u^n(s))|_H\,ds\right]=0.
\end{split}\end{equation*}
This completes the proof of Proposition \ref{prop-support}.
\end{proof}
Therefore,  the proof of Theorem \ref{thm-main} is now complete. \end{proof}

\section{Reflected Stochastic Navier Stokes Equations}\label{sec-Examples}
The main motivation of this paper is to treat the stochastic Navier Stokes Equations on a two dimensional domains $D$. Since we do not require any compactness of the embeddings, our domain can be unbounded, for example, the whole Euclidean space
$\mathbb{R}^2$. We will concentrate on this case and mostly follow a recent paper \cite{Brz+Ferrario_2019} by the first named author and Ferrario.

For a natural number $d$ and $p \in [1,\infty)$, let $L^p=L^p(\mathbb R^d,\mathbb R^d)$ be the classical Lebesgue space of all $\mathbb R^d$-valued Lebesgue measurable functions $v=(v^1,\ldots,v^d)$ defined on $\mathbb R^d$ endowed with the following classical  norm
\[
\|v\|_{L^p}=\left(\sum_{k=1}^d \|v^k\|_{L^p(\mathbb R^d)}^p\right)^{\frac1p}.
\]
For $p=\infty$, we set
$\|v\|_{L^\infty}=\max_{k=1}^d \|v^k\|_{L^\infty(\mathbb R^d)}$. \\
Set $J^s=(I-\Delta)^{\frac s2}$.
We define the generalized Sobolev spaces of divergence free vector
distributions, for $s \in \mathbb R$,  as
\begin{equation}\begin{aligned}
H^{s,p}&=\{u \in {\mathcal S}^\prime(\mathbb R^d,\mathbb{R}^d):
\|J^s u\|_{L^p}<\infty\},
\\
H^{s,p}_\sol&=\{u \in H^{s,p} : \divv  u =0 \}.
\end{aligned}\end{equation}
It is well known that $J^\sigma$ is an isomorphism between $H^{s,p}$ and $H^{s-\sigma,p}$
for $s \in \mathbb R$ and $1< p < \infty$.
Moreover $H^{s_2,p} \subset H^{s_1,p}$ when $s_1<s_2$.
In particular, for the Hilbert case $p=2$
we set $\rH=H^{0,2}_\sol$ and, for $s\neq 0$, $\rH^s=H^{s,2}_\sol$, so that (\textbf{Warning!}), $\rH^s$ is a proper closed subspace of the classical Sobolev space usually denoted by the same symbol.
In particular, we put
\[
\rH=\{v \in L^2(\mathbb R^d,\mathbb{R}^d): \divv  v =0\}
\]
with scalar product inherited from $L^2(\mathbb R^d,\mathbb{R}^d)$.

We will also denote by  $\langle\cdot, \cdot \rangle$ the
 duality bracket between $(H^{s,p})^\prime$ and $H^{s,p}$ spaces. Note that for $p \in [1,\infty)$, the space $(H^{s,p})^\prime$ can be identified with $(H^{-s,p^\ast})$, where $\frac1p+\frac1{p^\ast}=1$.

\dela{
Moreover,
$H_{\mathrm{loc}}$ is the  space $H$ with the topology
generated by the family of
semi-norms $\|v\|_{H_N}=\left(\int_{|\xi|<N}|v(\xi)|^2d\xi\right)^{1/2}$,
$N\in \mathbb N$, and $L^2(0,T;H_{\mathrm{loc}})$ is the space $L^2(0,T;H)$ with the topology
generated by the family of semi-norms $\|v\|_{L^2(0,T;H_N)}$,
$N\in \mathbb N$.
By $H_{\mathrm{w}}$  we denote the space $H$  with the weak topology and by
$C([0,T];H_{\mathrm{w}})$ the space of $H$-valued  weakly continuous
functions with the topology of uniform weak convergence on $[0,T]$;
in particular $v_n \to v$ in $C([0,T];H_{\mathrm{w}})$ means
\[
\lim_{n\to \infty} \sup_{0\le t\le T}|(v_n(t)-v(t),h)_H|=0
\]
for all $h \in H$.
Notice that $v(t) \in H$ for any $t$ if $v \in C([0,T];H_{\mathrm{w}})$.

From \cite{HW_1995} one knows that there exists a  separable Hilbert space $U$
such that $U$ is a dense subset
of $H^1$ and is compactly embedded in $H^1$.
We also have that
\[
U \subset H^1 \subset H \simeq H^\prime\subset H^{-1}\subset U^\prime
\]
with dense and continuous embedding, but in addition $H^{-1}$
is compactly embedded in $U^\prime$.}

Now we define the operators appearing in the abstract formulation. Assume that $s \in \mathbb R$ and  $1\le p<\infty$.
Let $A_0=-\Delta$; then
$A_0$ is a linear unbounded  operator in $H^{s,p}$ and bounded from  $H^{s+2,p}$ to $H^{s,p}$.
 Moreover, the spaces   $H^{s,p}_\sol$
 are invariant w.r.t. $A_0$  and the corresponding operator will be denoted by $A$. Let us observe that
 $A$ is a linear unbounded  operator in $H^{s,p}$ and bounded from  $H^{s+2,p}_\sol$ to $H^{s,p}_\sol$.
 The operator $-A_0$  generates
a contractive and analytic $C_0$-semigroup $\{e^{-tA}\}_{t\ge 0}$ on  $H^{s,p}$ and therefore,
the operator $-A$  generates
a contractive and analytic $C_0$-semigroup $\{e^{-tA}\}_{t\ge 0}$ on  $H^{s,p}_\sol$.
Moreover, for $t>0$ the operator $e^{-tA}$ is bounded
from $H^{s,p}_\sol$ into $H^{s^\prime,p}_\sol$
with $s^\prime>s$ and there exists a constant $M$ (depending on $s^\prime-s$ and $p$)
such that
\begin{equation}\label{semigruppo}
\|e^{-tA}  \|_{\mathcal L(H^{s,p}_\sol;H^{s^\prime,p}_\sol)}\le M (1+t^{-(s^\prime-s)/2}) .
\end{equation}
We have $A: H^1 \to  H^{-1}$ as a linear bounded operator and
\[
\langle Av,v\rangle=\|\nabla v\|_{L^2}^2, \;\;\; v\in H^1,
\]
where
\[
\|\nabla v \|_{L^2}^2=\sum_{k=1}^d \|\nabla v^k\|^2_{L^2}, \;\;\; v\in H^1.
\]

Moreover we have
\begin{equation}\label{quadrati}
\|v\|_{H^1}^2=\|v\|^2_{L^2}+\|\nabla v\|_{L^2}^2 .
\end{equation}

We define a bounded
trilinear form $b: H^1\times H^1 \times  H^{1}\to \mathbb{R}$ by
\[
b(u,v,z)=\int_{\mathbb R^d} (u(\xi)\cdot \nabla )v(\xi) \ \cdot z(\xi) \ d\xi,\;\; u,v,z \in H^1.
\]
and the corresponding bounded
 bilinear operator $B:H^1\times H^1\to H^{-1}$ via the
trilinear form
\[
\langle B(u,v),z\rangle= b(u,v,z) ,\;\; u,v,z \in H^1.
\]
This operator satisfies, for all $u,v,z \in H^1$,
\begin{equation}\label{scambio}
\langle B(u,v),z\rangle =-\langle B(u,z),v\rangle , \qquad
\langle B(u,v),v\rangle =0.
\end{equation}
Note that \eqref{scambio} implies a weaker version of it, i.e.
\begin{equation}\label{scambio-weak}
\langle B(u,u),u\rangle =0, \mbox{ for all $u \in H^1$}.
\end{equation}

$B$ can be extended to a bounded bilinear operator from
$H^{0,4}_\sol\times H^{0,4}_\sol$ to $H^{-1}$ with
\begin{equation}\label{bL4}
\|B(u,v)\|_{H^{-1}}\le \|u\|_{L^4}\|v\|_{L^4},
\end{equation}
see \cite[(2.29)]{Tenam1995}. Moreover, for any $a>\frac d2+1$, $B$ can be extended to a bounded bilinear operator from
$H\times H$ to $H^{-a}$ with
\begin{equation}\label{stimaB-a}
\|B(u,v)\|_{H^{-a}}\le C \|u\|_{L^2}\|v\|_{L^2},
\end{equation}
see \cite[Lemma 2.1]{Tenam1995}.
Let us also observe that in view of the so called Ladyzhenskaya inequality, see   \cite[Lemma 3.3.3]{Temam_2001},
inequality \eqref{bL4} implies that  there exists $C>0$ such that for all  $u\in H^1$

\begin{equation}\label{bL5}
\|B(u,u)\|_{H^{-1}}\le \sqrt{2} \vert u\vert_{L^2}^{\frac12} \vert \nabla u\vert _{L^2}^{\frac12} \vert u\vert _{L^2}^{\frac12} \vert \nabla u\vert _{L^2}^{\frac12}.
\end{equation}
In particular, this proves that inequality \eqref{1.1} in part (b) of Assumption \textbf{(A.2)}, or equivalently, inequality \eqref{1.2-b}.

Once and for all we denote by $C$ a generic constant, which may vary
from line to line; we number it if
we need to identify it.

Finally, we define the noise forcing term.
Given a real separable Hilbert space $\rK$
we consider a $\rK$-cylindrical Wiener process $\{W(t)\}_{t \ge 0}$ defined on a
stochastic basis $(\Omega,\mathbb F, \{\mathbb F_t\}_{t\ge 0},\mathbb P)$ satisfying the usual conditions.
For the covariance $\sigma$ of the noise we make the following assumptions.
\begin{enumerate}
	\item[{\bf (G1)}]
	the mapping
	$\sigma:H\to \gamma(\rK;H)$
	is well-defined and
\dela{	\[
	\sup_{v \in H} \|G(v)\|_{\gamma(Y;H^{-g})}=:K_{g,2}<\infty
	\]
	\item[{\bf (G2)}]
	the mapping  $G:H\to \gamma(Y;H^{- g ,4}_\sol)$
	is  well defined  and
	\[
	\sup_{v \in H} \|G(v)\|_{\gamma(Y;H^{-g,4}_\sol)}=:K_{g,4}<\infty
	\]
	\item[{\bf (G3)}] if assumption {\bf (G1)} holds, then for any
	$\phi \in H^{-g}$ the mapping $H\ni v\mapsto G(v)^* \phi \in Y$ is continuous
	when in  $H$ we consider the Fr\'echet topology inherited from
	the space $H_{\mathrm{loc}}$ or the weak topology of $H$
	\item[{\bf (G4)}] if assumption {\bf (G1)} holds,
	then $G$ extends to} is a Lipschitz continuous map
	$G:H\to \gamma(\rK;H)$, i.e.
	\[
	\exists\ L>0: \; \|\sigma(v_1)-\sigma(v_2)\|_{\gamma(\rK;H)}
	\le L\|v_1-v_2\|_{H}
	\]
	for all  $v_1, v_2 \in H$.
\end{enumerate}
\dela{We remind that the case $g=0$ in {\bf (G1)} would give a
Hilbert-Schmidt operator in $H$, which is the case considered in
\cite{Brz+Mot_2013,Brz+Ondr+Motyl_2016}.
}

We consider the  stochastic damped Navier-Stokes equations,
that is the equations of motion of a viscous
incompressible fluid with two forcing terms, one is random  and the
other one is deterministic. These equations are
\begin{equation}\label{sist:ini}
\begin{cases}
\partial_t u+[-\nu \Delta u +\gamma u+(u \cdot \nabla)u +\nabla p]\ dt
= \sigma(u)\ \partial_t W+f\ dt
\\
\divv  u=0
\end{cases}
\end{equation}
where the unknowns are the vector velocity $u=u(t,\xi)$
and  the scalar pressure $p=p(t,\xi)$ for $t\ge 0$ and  $\xi \in \mathbb R^d$.
By $\nu>0$ we denote  the kinematic viscosity  and by $\gamma> 0$  the
 sticky  viscosity, see for instance \cite{Gal_2002} and  \cite{CR}. When
 $\gamma=0$ \eqref{sist:ini} reduce to the classical stochastic Navier-Stokes equations.
The notation $ \partial_t W$ on the right hand side  is  for the space correlated  and  white in time
noise and $f$ is a deterministic forcing term.  We consider a
multiplicative term
$\sigma(u)$ keeping track of the fact that the noise may depend on the
velocity.

Projecting equations \eqref{sist:ini} onto the space $H$  of divergence
free vector fields, we get
the  abstract form of the stochastic damped
Navier-Stokes equations \eqref{sist:ini}
\begin{equation}\label{sns}
du(t)+[Au(t)+\gamma u(t)+B\left(u(t),u(t)\right)]\,dt
=
\sigma(u(t))\,dW(t)+f(t)\ dt
\end{equation}
with the initial condition
\begin{equation}\label{sns-ic}
u(0)=u_0
\end{equation}
where the initial velocity $u_0:\Omega\to H$ is  an $\mathcal{F}_0$-measurable random variable .
Here $\gamma>0$ is fixed  and for simplicity we have put $\nu=1$.
The case $\gamma=0$ was considered in, e.g.   \cite{Brzezniak+Ferrario_2017}.
We assume that $\gamma>0$ so that the operator $A+\gamma I$ satisfies the assumption \eqref{eqn-Poincare}. However, one can slightly modify the proofs in the paper in such a way that assumption \eqref{eqn-Poincare} is replaced by  the following one. There exists $\alpha>0$  such that
\begin{equation}\label{eqn-Poincare-weak}
( Au,u)+\alpha ( u,u)  \geq 0, \;\;\; u \in D(A).
\end{equation}
 Here we consider classical strong solutions.

\begin{definition}	[strong  solution]
	We say that
a progressively measurable process $u:[0,T]\times \Omega\to H$ with
		${\mathbb P}$-almost all  paths satisfying
		\[
		u \in C([0,T];H)\cap L^2(0,T;H^1)
		\]
is a strong solution to problem  \eqref{sns} if and only if
	\begin{enumerate}
		\item[$\bullet$]  $ u(0)=u_0$
		\item[$\bullet$]
 for any $t \in [0,T], \psi \in H^2$, ${\mathbb P}$-a.s.,
		\begin{equation}
\label{sol-path}
\begin{aligned}
		( u(t),\psi)_H
		&+\int_{0}^t \langle Au(s) ,\psi\rangle ds
		+\gamma \int_{0}^t ( u(s) ,\psi)_H ds
		\\\nonumber
		&+\int_{0}^t \langle B(u(s),u(s)),\psi \rangle ds
		\\
		&=
		( u(0),\psi)_H
		+\int_{0}^t \langle f(s) ,\psi\rangle ds
		+\langle \int_{0}^t \sigma(u(s))\,d W(s),\psi\rangle.
		\end{aligned}\end{equation}
	\end{enumerate}
\end{definition}
Now consider the reflected stochastic Navier-Stokes equation:
\begin{equation}\label{reflected-Navier-Stokes}
du(t)+[Au(t)+\gamma u(t)+B\left(u(t),u(t)\right)]\,dt
=
\sigma(u(t))\,dW(t)+f(t)\ dt+dL(t)
\end{equation}
with the initial condition
\begin{equation}\label{sns-ic-2}
u(0)=u_0\in \bar{D}
\end{equation}

Our main result, Theorem 4.1,  applies in this setting. We have the following
	\begin{theorem}\label{Navier-Stokes}
The reflected stochastic evolution equation (\ref{reflected-Navier-Stokes})  admits a unique solution $(u, L)$ that satisfies, for $T> 0$,
\begin{equation}\label{2.15-b}
\mathbb{E}\,\bigl[\sup_{t \in [0,T]}|u(t)|_H^2\, +\, \int_0^T \Vert u(t) \Vert^2\,dt\,\bigr]< \infty.
\end{equation}
\end{theorem}

\noindent
\dela{All the terms in \eqref{sol-path} make sense; in particular
the trilinear term is well defined thanks to \eqref{bL4} which
provides the following estimate
\[
\Big|\int_{0}^t \langle B(\hat v(s),\hat v(s)),\psi \rangle ds\Big|
\le
\|\psi\|_{H^1}\int_0^t\|\hat v(s)\|_{L^4}^2ds.
\]
Let us point out that in the following sections we shall prove the existence of a martingale
solution with a time integrability higher than
$\hat v \in L^2(0,T;L^4)$.
As far as the stochastic integral is concerned, by assuming {\bf (G1)}
we obtain that the random variable
$\int_{0}^t G(\hat v(s))\,d\hat w(s)$ belongs to
$H^{-g}$ (${\mathbb P}$-a.s.), see
Proposition 5.2 in \cite{Brzezniak+Ferrario_2017}.
\\
Initial deterministic velocity   $x \in H$ corresponds to $\mu=\delta_x$.
}

\vskip 0.5cm
\noindent {\bf Acknowledgement}. This work is partly supported by the National Natural Science Foundation of China (NSFC) (No. 12131019, No. 11721101).

\end{document}